\documentclass[11pt,reqno]{amsart}

\usepackage{enumerate, amsmath, amsthm, amsfonts, amssymb, 
mathrsfs, graphicx, paralist, ulem, float}
\usepackage[usenames, dvipsnames]{color}
\usepackage[margin=1in]{geometry} 
\usepackage[colorlinks, linktoc=page, linkcolor={red}, citecolor={green!65!black}, urlcolor={blue}]{hyperref}
\usepackage{comment}
\usepackage{bbm}
\usepackage{dynkin-diagrams}
\usepackage{wrapfig}
\usepackage{mathtools}
\usepackage{mathabx} 
\usepackage{breqn} 
\usepackage{tikz-cd}
\usepackage{amsrefs}
\usepackage{listings}

\usetikzlibrary{angles}
\usetikzlibrary{arrows.meta} 
\usetikzlibrary{decorations}
\usetikzlibrary{calc}
\usetikzlibrary{cd}
\usetikzlibrary{backgrounds}
\usepackage{quiver}
\tikzcdset{scale cd/.style={every label/.append style={scale=#1},
    cells={nodes={scale=#1}}}}

\usepackage{bbm}

\numberwithin{equation}{section}
\newtheorem{Theorem}[equation]{Theorem}
\newtheorem{Proposition}[equation]{Proposition}
\newtheorem{Lemma}[equation]{Lemma}
\newtheorem{Corollary}[equation]{Corollary}
\newtheorem{Conjecture}[equation]{Conjecture}

\theoremstyle{definition}
\newtheorem{Remark}[equation]{Remark}

\newtheorem{eg}[equation]{Example}
\newenvironment{example}[1][]{\begin{eg}[#1] \pushQED{\qed}}{\popQED \end{eg}}
\newtheorem{Definition}[equation]{Definition}

\newtheorem{Notation}[equation]{Notation}

\newcommand{\sD}{\mathscr{D}}

\newcommand{\cF}{\mathcal{F}}

\newcommand{\cG}{\mathcal{G}}

\newcommand{\sG}{\mathscr{G}}

\newcommand{\cH}{\mathcal{H}}

\newcommand{\cQ}{\mathcal{Q}}

\newcommand{\cR}{\mathcal{R}}

\newcommand{\sR}{\mathscr{R}}

\newcommand{\cX}{\mathcal{X}}

\newcommand{\CC}{\mathbb{C}}

\newcommand{\RR}{\mathbb{R}}

\renewcommand{\phi}{\varphi}
\renewcommand{\emptyset}{\varnothing}

\renewcommand{\tilde}[1]{\widetilde{#1}}
\newcommand{\ol}[1]{\overline{#1}}
\newcommand{\ul}[1]{\underline{#1}}
\newcommand{\wt}[1]{\widetilde{#1}}

\newcommand{\oded}[1]{\textcolor{red}{$[\star$ Oded: #1 $\star]$}}

\newcommand{\owen}[1]{\textcolor{green}{$[\star$ Owen: #1 $\star]$}}

\renewcommand{\hom}{\operatorname{Hom}}

\DeclareMathOperator{\norm}{Norm}

\DeclareMathOperator{\End}{End}

\DeclareMathOperator{\Aut}{Aut}

\DeclareMathOperator{\Cone}{Cone}

\DeclareMathOperator{\Deck}{Deck}

\DeclareMathOperator{\Cham}{Cham}

\newcommand{\BH}{\mathscr{BH}} 
\newcommand{\D}{\mathscr{D}} 
\newcommand{\R}{\mathscr{R}} 

\makeatletter
\newcommand{\xMapsto}[2][]{\ext@arrow 0599{\Mapstofill@}{#1}{#2}}
\def\Mapstofill@{\arrowfill@{\Mapstochar\Relbar}\Relbar\Rightarrow}
\makeatother

\newcommand{\Span}{\mathrm{Span}}

\newcommand\iso\cong
\newcommand\into\hookrightarrow
\newcommand\onto\twoheadrightarrow
\newcommand{\Hom}{\operatorname{Hom}}
\newcommand{\nc}{\newcommand}
\nc{\la}{\lambda}
\nc{\Iso}{\mathsf{Iso}}
\nc{\Id}{\mathrm{Id}}

\begin{document}

\title[Normalisers and hyperplane arrangements]{Normalisers of parabolic subgroups of Artin--Tits groups  \\ and Tits cone intersections}

\author{Owen Garnier}
\address{O.~Garnier: Departamento de \'Algebra e Instituto de Matem\'aticas de la Universidad de Sevilla, Spain}
\email{owen.garnier@math.cnrs.fr}
\author{Edmund Heng}
\address{E.~Heng: School of Mathematics and Statistics, University of Sydney, Australia}
\email{edmund.heng@sydney.edu.au}
\author{Anthony Licata}
\address{A.~Licata: CNRS-ANU International Research Laboratory FAMSI and Mathematical Sciences Institute, The Australian National University, Australia}
\email{anthony.licata@anu.edu.au}
\author{Oded Yacobi}
\address{O.~Yacobi: School of Mathematics and Statistics, University of Sydney, Australia}
\email{oded.yacobi@sydney.edu.au}
\date{}

\begin{abstract}
Let $(W,S)$ be a Coxeter system and let $J \subseteq S$. Motivated by 3-fold flops, Iyama and Wemyss study the hyperplane arrangement in the Tits cone intersection of $J$, which is a $J$-relative generalisation of the classical Coxeter arrangement. 
For $\Gamma$ of spherical type, we show that the complexified complement of this $J$-relative hyperplane arrangement is a $K(\pi,1)$ space for the normaliser (quotient) of the standard parabolic subgroup attached to $J$ in the Artin--Tits group.
For general $\Gamma$ we show that the Brink--Howlett groupoid, which describes normalisers of parabolic subgroups of the Coxeter groups, has its universal cover described by the wall-and-chamber structure of the Tits cone intersection. We use this to show that wall crossing sequences satisfy an ``atomic Matsumoto relation'', generalising a theorem of Ko and answering questions raised by Iyama and Wemyss.



\end{abstract}

\maketitle

\section{Introduction}
\subsection{General overview}
Let $(W,S)$ be a Coxeter system and consider the associated hyperplane arrangement $(\Cone^\circ, \cH)$. Here $\Cone^\circ$ denotes the interior of the Tits cone and $\cH$ is the set of reflecting hyperplanes.  Much of the structure of $W$ is reflected in the geometry of the arrangement $(\Cone^\circ, \cH)$, and this geometry also
plays an important role in the study of the Artin--Tits group $A$ of $W$.
Indeed, the complexified hyperplane complement $\cX^\CC$ associated to $(\Cone^\circ, \cH)$ (see \eqref{eqn:complexfiedTitscone}) carries a free and properly discontinuous action of $W$, and the quotient $\cX^\CC\to \cX^\CC/W$ is a normal covering inducing the short exact sequence:
\begin{align} \label{eqn:introclassicalses}
1 \to \pi_1(\cX^\CC,x) \to \pi_1(\cX^\CC/W,x) \to W \to 1.
\end{align}
We then have isomorphisms $A\cong \pi_1(\cX^\CC/W,x)$ and $P\cong \pi_1(\cX^\CC,x)$ ($P$ is the pure Artin--Tits group). 
Conjecturally, $\cX^\CC/W$ is a classifying space for $A$.
This statement -- known as the $K(\pi,1)$ conjecture -- is proven  for spherical types \cite{Del72}, FC-types \cite{CD_FCtype}, affine types \cite{Paolini-Salvetti}, and more \cites{DPS_rank3,Huang_4cycles,HH_productwithZ}, but remains open in general.

Now let $J\subseteq S$.  The main goal of this paper is to generalise and study $J$-relative versions of the groups and arrangements associated to $(W,S)$.  In this relative setting, the groups $W$ and $A$ are replaced by normaliser quotients of the standard parabolic subgroups associated to $J$, and the Coxeter arrangement $(\Cone^\circ, \cH)$ is replaced by an arrangement restricted to the \textit{Tits cone intersection} of $J$ introduced by Iyama--Wemyss \cite{IW}.

\subsubsection{Normaliser quotients of parabolics}
The choice $J$ defines standard parabolic subgroups $W_J \subseteq W$ and $A_J\subseteq A$,
normaliser subgroups $\mathrm{Norm}_W(W_J), \mathrm{Norm}_A(A_J)$ and normaliser quotients:
\begin{equation}\label{def:N(W,J)_and_N(A,J)}
N(W,J)\coloneq \mathrm{Norm}_W(W_J)/W_J\quad \text{ and } \quad N(A,J)\coloneq \mathrm{Norm}_A(A_J)/A_J.
\end{equation}
There is significant literature investigating these groups: the papers \cites{Lusztig76, Howlett80, Borcherds} initiate the study of normaliser quotients of Coxeter groups.  In particular, Brink and Howlett introduce a groupoid $\BH$ in \cite{BH} (see Section \ref{sec:BHgroupoid}) and use it to present and study $N(W,J)$.
For Artin--Tits groups, normaliser quotients were studied in \cites{Paris, GodelleI, Godelle_RibbonII} for spherical types,  \cite{Godelle_FCparabolic} for FC type,  \cite{Godelle_CAT(0)} for two-dimensional type, and \cite{Haettel_injectivemetric} for affine types $\tilde{A}$ and $\tilde{C}$.

\subsubsection{Tits cone intersection}
On the geometric side, the Tits cone intersection of $J$ is a convex cone $\Cone(J)\subseteq \Cone$, and we can consider the restriction $\cH^J$ of $\cH$ to this subspace; see Section \ref{sec:Jdef} for the precise definitions. 
As in the classical case, the hyperplane arrangement in the interior $(\Cone(J)^\circ, \cH^J)$ is locally finite (see Proposition \ref{prop:locallyfinite}), and when $J=\emptyset$ the arrangement $(\Cone(J)^\circ,\cH^J)$ recovers the classical Coxeter arrangement $(\Cone^\circ,\cH)$. 

The hyperplane arrangement $(\Cone(J)^\circ,\cH^J)$ plays an important role in the homological minimal model program and the theory of 3-fold flops \cite{Wemyss_flopscluster}. 
Here, each chamber of $(\Cone(J)^\circ,\cH^J)$ labels the derived category of a minimal model of a 3-fold flopping contraction; or algebraically, its corresponding non-commutative crepant resolution. 
Adjacent chambers are related by flopping equivalences and geodesic paths with the same source and target give rise to compositions of flopping functors that are isomorphic. 
As such, the fundamental group of the complexified hyperplane complement $\cX_J^\CC$ (see \eqref{eqn:complexifiedJcone}) acts on the derived categories via composition of these flopping equivalences. 
For more details, we refer the reader to \cites{DW_braidsflops,HW_faithfulhyperplane,HW_stability3fold} for the geometric setting, and \cites{AW_contractionalg,IW} for the algebraic setting.
Motivated by this, Iyama and Wemyss classify these arrangements in low ranks \cite{IW}, noting that they are not in general Coxeter arrangements (see Figure \ref{fig:TCI-arrangement} in \S\ref{sec:example}). 

\subsection{Main results}
Just as the Coxeter arrangement $(\Cone^\circ,\cH)$ controls the Coxeter group $W$ and Artin--Tits group $A$, our results show that the Tits cone intersection $(\Cone(J)^\circ,\cH^J)$ controls the normaliser quotients:
\begin{itemize}
\item In arbitrary type, we show that the groupoid $\BH$ 
can be obtained geometrically from the arrangement $(\Cone(J)^\circ,\cH^J)$. We also prove an atomic Matsumoto theorem for $\BH$.
\item In spherical type, we show that the $N(W,J)$-quotient of the complexified hyperplane complement $\cX_J^\CC$ associated to $(\Cone(J)^\circ,\cH^J)$ is a $K(\pi,1)$ space for $N(A,J)$. 
\end{itemize}
We now state the precise theorems.

\subsubsection{Brink--Howlett groupoids via Tits cone intersections}
For our first main theorem, we recall that, in direct analogy with topological spaces, there are notions of groupoid coverings, deck transformations, etc.\ (see Section \ref{sec:groupoid_coverings}).
In particular, any connected groupoid $\sG$ admits a universal groupoid covering $\wt{\sG} \to \sG$, whose deck transformation group is the vertex group of $\sG$.

\begin{Theorem}[= Theorem \ref{thm:WCuniversalcover}]\label{thm:main2v2}
The universal cover of $\BH$ is identified with the wall-crossing groupoid $\wt\BH$ of the Tits cone intersection: the objects of $\wt\BH$ are the chambers of $(\Cone(J)^\circ,\cH^J)$, and its generating morphisms are simple wall crossings between adjacent chambers.
Moreover, $N(W,J)$ is the deck transformation group of this covering, so that $\BH$ can be obtained as the $N(W,J)$-quotient $\wt\BH$.
\end{Theorem}
We emphasise that our theorem above is independent of the results of Brink--Howlett in \cite{BH}: for example, it is a consequence of our theorem that $\BH$ has vertex group $N(W,J)$. 
In fact, we obtain from this new proofs of some known results related to $N(W,J)$ and $\BH$; see Corollary \ref{cor:equivalence_brink_howlett} and \ref{cor:BHthmA}.
Moreover, the groupoid covering in Theorem \ref{thm:main2v2} sends the generators in $\wt{\BH}$ (simple wall crossings) to the generators of $\BH$ considered by Brink and Howlett in \cite{BH}.
As a result, standard expressions in $\BH$ (analogs of reduced expressions in $W$; see Definition \ref{def:stdexp}) are in bijection with geodesic paths in $(\Cone(J)^\circ,\cH^J)$; see Lemma \ref{lem:stdexp-geodesic} for more details.


Brink and Howlett also showed in \cite{BH} that $\BH$ carries a groupoid presentation that is formally parallel to the Coxeter presentation of $W$: its relations are split into quadratic relations and braid-like relations known as the \textit{atomic braid relations} (see Section \ref{sec:atomicM} and Theorem \ref{thm:BHthmA}). 
In the case of Coxeter groups, the famous Matsumoto theorem states that any two standard expressions (i.e.\ reduced expressions) of a given element are related by a sequence of braid relations, without needing to use quadratic relations. Our second main result is a similar statement at the level of Brink--Howlett groupoids $\BH$ (and its universal cover):
\begin{Theorem}[= Theorem \ref{thm:atomicMat}, ``Atomic Matsumoto theorem'']\label{thm:atomicMatintro}
    Any two standard expressions of a morphism in $\BH$ are related by a sequence of atomic braid relations.
\end{Theorem}
This generalises \cite[Theorem 1.8]{Ko25} beyond the finitary case; see Remark \ref{rem:relation_to_ko}. Translating back to $\wt{\BH}$ via Theorem \ref{thm:main2v2} and Lemma \ref{lem:stdexp-geodesic} gives Corollary \ref{cor:arrangement_atomicMat}: any two geodesic paths in $(\Cone(J)^\circ,\cH^J)$ with the same endpoints are related by a finite sequence of codimension-two relations (see Definition \ref{def:codim2rel}). This answers a question raised by Iyama--Wemyss \cite[Remark 1.62]{IW}.

\subsubsection{\texorpdfstring{$K(\pi,1)$}{K(pi,1)} spaces via Tits cone intersections}
Analogous to the classical setting, the complexified hyperplane complement $\cX_J^\CC$ (see \eqref{eqn:complexifiedJcone}) associated to $(\Cone(J)^\circ, \cH^J)$ carries a free and properly discontinuous action of the normaliser quotient $N(W,J)$. Therefore $\cX_J^\CC\to \cX_J^\CC/N(W,J)$ is a normal covering inducing the short exact sequence (cf.\ \eqref{eqn:introclassicalses})
\begin{align}\label{eq:ses3}
1 \to \pi_1(\cX_J^\CC,x) \to \pi_1(\cX_J^\CC/N(W,J),x) \to N(W,J) \to 1.
\end{align}
The algebraic counterpart to \eqref{eq:ses3} is the short exact sequence (Lemma \ref{lem:PAW})
\begin{align}\label{eq:ses4}
1 \to N(P,J) \to N(A,J) \to N(W,J) \to 1,
\end{align}
where $N(P,J)$ is defined as the image of $P\cap \mathrm{Norm}_A(A_J)$ in $N(A,J)$:
\begin{align}
        N(P,J) \coloneqq \{\beta A_J \in N(A,J) \mid P\cap \beta A_J \neq \emptyset\}. \label{def:N(P,J)}
\end{align}
Our $K(\pi,1)$ theorem says that, in spherical type, the geometric and algebraic sequences are isomorphic.
\begin{Theorem}\label{thm:main1v2}
Assume that $W$ is a finite Coxeter group, and let $x\in \cX_J^\CC$. The short exact sequences \eqref{eq:ses3} and \eqref{eq:ses4} are isomorphic. In particular, we have the following isomorphisms:
    \begin{align*}
        N(A,J) &\cong \pi_1(\cX_J^\CC/N(W,J),x), \\ 
        N(P,J) &\cong \pi_1(\cX_J^\CC,x).
    \end{align*}
    Moreover, $\cX^\CC_J/N(W,J)$ (resp. $\cX^\CC_J$) is a $K(\pi,1)$ space for $N(A,J)$ (resp. for $N(P,J)$).
\end{Theorem}
On the one hand, this theorem provides an algebraic description of the fundamental group of $\cX^\CC_J$ in terms of the pure Artin--Tits group $P$.
On the other hand, it shows that the normaliser quotient $N(A,J)$ also has a finite CW complex as its $K(\pi,1)$, since $\cX^\CC_J/N(W,J)$ is homotopy equivalent to a finite CW complex; see Remark \ref{rmk:finiteCWcomplex}.

The proof of Theorem \ref{thm:main1v2} appears in Section \ref{sec:mainresult} (Theorem \ref{thm:main} and Corollary \ref{cor:Kpi1}).
Our proof uses two groupoids.
The first is the Deligne groupoid $\D$ associated to $(\Cone(J)^\circ,\cH^J)$, which is equivalent to the fundamental groupoid $\pi_1(\cX_J^\CC)$ (see Section \ref{sec:Dgroupoid}). 
The second is the reduced ribbon groupoid $\R$ (associated to $A$ and $J$) introduced by Godelle \cite{Godelle_RibbonII}, which can be thought of as a lift of the Brink--Howlett groupoid $\BH$ from $W$ to $A$ (see Definition \ref{def:reduced_ribbon_groupoid}).
In particular, the vertex group of $\R$ is isomorphic to $N(A,J)$; see Section \ref{sec:reducedribgroupoid} for more details.
We obtain Theorem \ref{thm:main1v2} by proving that $\pi_1(\cX_J^\CC/N(W,J)) \cong \D/N(W,J)$ is equivalent to $\R$.

The relations between the groupoids considered in this paper are summarised by the following commutative diagram (cf. Remark \ref{rem:construction_diagram_global}):
\begin{equation}\label{eq:global_diagram}\begin{tikzcd}[row sep = small]
	{\pi_1(\mathcal{X}_J^{\mathbb{C}},x)} & {\mathscr{D}} && {\widetilde{\mathscr{BH}}} & 1 \\
	\\
	{N(A,J)} & {\mathscr{R}} && {\mathscr{BH}} & {N(W,J)}
	\arrow[squiggly, no head, from=1-2, to=1-1]
	\arrow[from=1-2, to=1-4]
	\arrow["{\cG}", two heads, from=1-2, to=3-2]
	\arrow[squiggly, no head, from=1-4, to=1-5]
	\arrow[two heads, from=1-4, to=3-4]
	\arrow[squiggly, no head, from=3-2, to=3-1]
	\arrow["{A\to W}"', from=3-2, to=3-4]
	\arrow[squiggly, no head, from=3-4, to=3-5]
\end{tikzcd}\end{equation}
The wavy lines indicate that the groupoid $\mathscr{G}$ is equivalent to the group $G$ (or equivalently that any vertex group in $\mathscr{G}$ is isomorphic to $G$). The two vertical arrows are normal groupoid coverings with deck transformation group $N(W,J)$. The top horizontal arrow is a natural functor detailed in Remark \ref{rem:obvious_functor}, and the bottom horizontal arrow is the functor induced by the epimorphism $A\to W$. 

We emphasise  that this whole diagram is defined only when $W$ is finite (since $\R$ is only defined when $W$ is finite), but the functors $\D \to \widetilde{\BH} \to \BH$ are defined in general. Further, we note that the bottom row is essentially due to work of Godelle, and our work relates $\sR$ and $\BH$ directly to the geometrically defined groupoids $\D$ and $\wt{\BH}$.
In Section \ref{sec:example} we work through an explicit example to illustrate this.

Finally, we conclude with a $K(\pi,1)$-conjecture for normaliser quotients of Artin--Tits groups:
\begin{Conjecture}
    Theorem \ref{thm:main1v2} holds for any Coxeter group $W$.
\end{Conjecture}

\subsection*{Acknowledgements}
We thank Juan Gonz\'{a}lez-Meneses, Bob Howlett, Luis Paris and Michael Wemyss for useful discussions.
We also thank the organisers of ParisFest, where part of this work was initiated. We are grateful to an anonymous referee whose careful reading of our manuscript greatly improved it. This work was supported by the Australian Research Council grant DP230100654. The first author was supported by the research project PID2022-138719NA-I00 financed by the Spanish Ministry of Science and Innovation. 
The third author acknowledges the FAMSI collaboration between the Australian National University and CNRS.

\section{Background}
\subsection{Hyperplane arrangements}
Let $V$ be a real vector space.
A hyperplane will always refer to a linear hyperplane (so all hyperplanes pass through the origin).
Throughout this paper, a \textbf{hyperplane arrangement} consists of a convex subset $T \subseteq V$ and a (possibly infinite) set $\cH$ of hyperplanes in $V$ such that $H\cap T \neq \emptyset$ for all $H \in \cH$.
We denote a hyperplane arrangement by the pair $(T,\cH)$.
Given a hyperplane arrangement $(T,\cH)$ and a convex subset $T' \subseteq T$, we can define a restricted hyperplane arrangement in $T'$ by removing from $\cH$ all the hyperplanes that do not intersect with $T'$.
For notation simplicity, we simply denote this arrangement by $(T',\cH)$, with the understanding that only hyperplanes in $\cH$ that intersect $T'$ are considered.
Finally, we say that a hyperplane arrangement is locally finite if for all $x \in T$, there exists an open neighbourhood $U_x$ of $x$ in $T$ such that $U_x$ only intersects finitely many hyperplanes in $\cH$.

Given a hyperplane arrangement $(T,\cH)$, the connected components of $T \setminus \bigcup_{H \in \cH} H$ are called \textbf{chambers}. 
Following \cite[Definition 1.27]{IW}, a \textbf{wall} of a chamber $C$ is the intersection of $\overline{C}$ with a hyperplane $H \in \cH$ that has codimension $1$.
(Note that the support of a face is called a wall in \cites{Bourbaki_4-6, Del72, Paris_Kpi1}, which we have reserved for a different usage above.)

\subsection{Coxeter and Artin--Tits groups}In this section we give general notation and results which will be used throughout the paper.

Let $\Gamma$ be a Coxeter diagram with vertex set $\Gamma_0$. 
Let $W$ denote the associated Coxeter group, with the set of simple generators $S = \{s_i \mid i \in \Gamma_0\}$, so that $(W,S)$ is a Coxeter system. 
We say that $\Gamma$ is \textbf{spherical} (or finite-type), if $W$ is finite. Note that we only impose the spherical-type assumption in Section \ref{sec:reducedribgroupoid} and \ref{sec:mainresult}.

We denote the simple roots of $\Gamma$ by $\{\alpha_i \mid i \in \Gamma_0\}$, 
and let $V = \bigoplus_{i \in \Gamma_0}\RR\alpha_i$ be the standard geometric representation of $W$.
The set of roots $\Phi = \{w \cdot \alpha_i \mid w \in W, i\in \Gamma_0\}$ decomposes into positive and negative roots, which we denote by $\Phi^+$ and $\Phi^-$ respectively. For $x \in W$, we let $\ell(x)$ denote the length of $x$ with respect to the simple generators. That is $\ell(x) = |\{\alpha \in \Phi^+ \mid w\cdot \alpha \in \Phi^- \}|$. A \textbf{left-divisor} of $w \in W$ is an element $u \in W$ such that $u \leq w$ with respect to the left weak Bruhat order on $W$.

Following \cite{IW}, we use $\Theta=V^*$ to denote the contragredient representation of $W$ over $\RR$. 
Recall that $\Theta$ realises $W$ as a reflection group in the sense of Vinberg \cite{Vinberg_reflectiongroup}.
For $\alpha \in \Phi$, consider the hyperplanes:
\begin{align*}
    H_\alpha &= \{ \varphi\in \Theta \mid \varphi(\alpha)=0 \},
\end{align*}
and let $\cH=\{H_\alpha \mid \alpha \in \Phi \}$ denote the set of (real) hyperplanes. 
 
We will be particularly interested in hyperplane arrangements and their associated complexified complements. To define these, first consider the open fundamental chamber given by $$C = \{ \varphi \in \Theta \mid \varphi(\alpha_i) > 0 \text{ for all } i \in \Gamma_0 \}.$$
The \textbf{Tits cone} is defined as 
\[
\Cone \coloneqq \bigcup_{w \in W} w\cdot \overline{C}.
\]
Let $\Cone^\circ$ denote the interior of $\Cone \subseteq \Theta$. 
The pair $(\Cone^{\circ},\cH)$ is a locally-finite hyperplane arrangement (see also Remark \ref{rem:locallyfinite}). Note that $W$ is finite if and only if $\Cone= \Theta$.

\begin{Notation}\label{not:complexified_complement}
We will  refer to the space
\begin{equation} \label{eqn:complexfiedTitscone}
\cX^\CC \coloneq \Big(\Cone^\circ \times \Cone^\circ\Big) \setminus \bigcup_{H \in \cH}H\times H.
\end{equation}
as the \textbf{complexified hyperplane complement} of the arrangement $(\Cone^\circ,\cH)$. This slight abuse of terminology is justified by the fact that when $W$ is finite, we have
\begin{equation}\label{eq:XCC}
    \cX^\CC \cong \Big(\Theta \otimes_\RR \CC\Big) \setminus \bigcup_{H \in \cH} H \otimes_\RR \CC.
\end{equation}
and $\cX^\CC$ is then the actual complexified complement of $(\Cone^\circ,\cH)=(\Theta,\cH)$.
\end{Notation}

In general,  the action of $W$ on $\cX^\CC$ is free and properly discontinuous, and the normal covering map $\cX^\CC \to \cX^\CC/W$ induces the following short exact sequence (where $Z\in \cX^\CC$):
\begin{equation} \label{eqn:classicalSES}
1 \to \pi_1(\cX^\CC,Z) \to \pi_1(\cX^\CC/W,Z) \xrightarrow{\pi} W \to 1.
\end{equation}

We define the \textbf{Artin--Tits group} $A \coloneq \pi_1(\cX^\CC/W,Z)$, and the \textbf{pure Artin--Tits group} is $P\coloneq\pi_1(\cX^\CC,Z)$. Note that the surjection $\pi:A \to W$ admits a set-theoretic section $w \mapsto \sigma_w$, the \textbf{positive lift} of $w$. We note that $A$ also has a presentation with Artin generators denoted $\sigma_i:=\sigma_{s_i},\; i \in \Gamma_0$. We won't be explicitly using the presentation here, and refer the reader to \cite{VdL_thesis} for more details. 

\subsection{Normalisers of parabolic subgroups}We keep the notation from the previous section, and we fix a subset $J$ of $\Gamma_0$ for this section. 

Let $W_J\coloneq\langle s_j \mid j \in J\rangle$ (resp. $A_J\coloneq\langle \sigma_j \mid j \in J\rangle$) denote the standard parabolic subgroup of $W$ (resp. of $A$) attached to $J$. Let $\Phi_J$ be the subset of roots of coming from $W_J$:
$$
\Phi_J \coloneqq \{ w \cdot \alpha_j \mid j \in J,\; w \in W_J \} \subseteq \Phi,
$$
and denote $\Phi_J^\pm = \Phi_J \cap \Phi^\pm$. 

\begin{Notation}\label{notation:longest_element}
Let $J \subseteq \Gamma_0$ be such that $W_J$ is finite. The group $W_J$ contains a unique longest element for the length function $\ell$. We denote it by $w_J \in W_J$, and we denote by $\Delta_J \in A_J$ its positive lift to the Artin--Tits group. The element $w_J$ induces a bijection 
\begin{equation} \label{eqn:diagraminvolution}
\iota_J:J \xrightarrow{\equiv} J
\end{equation}
defined by the formula $w_J\cdot\alpha_j = -\alpha_{\iota(j)}$.  
\end{Notation}

\begin{Definition}\label{def:associate}
Let $\Gamma$ be a Coxeter diagram. An \textbf{associate} of $J\subseteq \Gamma_0$ is a subset $I \subseteq \Gamma_0$ such that $W_I$ and $W_J$ are conjugate in $W$. 
\end{Definition}

The normalisers of parabolic subgroups and their quotients will be of primary interest to us, and we recall the definitions of $N(W,J), N(A,J)$ and $N(P,J)$ from \eqref{def:N(W,J)_and_N(A,J)} and \eqref{def:N(P,J)}.



For later use, we record here an easy lemma relating the groups above.
\begin{Lemma}\label{lem:PAW}
    The map $\pi:A \to W$ induces a map $\ol{\pi}: N(A,J) \to N(W,J)$ giving the short exact sequence
    \begin{equation}\label{eq:ses2}
1 \to N(P,J) \to N(A,J) \xrightarrow{\ol{\pi}} N(W,J) \to 1.
\end{equation}
\end{Lemma}

\begin{proof}
Since $\pi$ is surjective, and since $\pi(A_J)=W_J$, we can consider the restricted morphism $\mathrm{Norm}_A(A_J)\to \mathrm{Norm}_W(W_J)$, along with its composition with the projection map $\mathrm{Norm}_W(W_J)\to N(W,J)$. The kernel of this composed morphism contains $A_J$ and we obtain an induced morphism $\ol{\pi}:N(A,J)\to N(W,J)$.

Let $\beta A_J \in N(A,J)$ and suppose $\beta A_J \in \ker(\ol{\pi})$, i.e. $\ol{\pi}(\beta A_J) = W_J$. We have $w \coloneq \pi(\beta) \in W_J$, and hence $\beta A_J = \beta \sigma_w^{-1} A_J$. We also have that $\beta \sigma_w^{-1}$ lies in the pure Artin--Tits group $P$. This shows that $\beta A_J \in N(P,J)$ (see \eqref{def:N(P,J)}), and hence $\ker(\ol{\pi}) \subseteq N(P,J)$. The other inclusion is clear. \end{proof}

\subsection{Groupoid coverings}\label{sec:groupoid_coverings}
The theory of groupoid coverings is a direct translation of the classical theory of covering spaces (see e.g.\ \cite[Section 10.2]{brown_groupoids}), and we recall the basics here. 

A \textbf{groupoid} $\sG$ is a category in which every morphism is invertible. For us, all groupoids are assumed to be connected and small. For $x \in \sG$, we let $\sG_x \coloneq \End_\sG(x)$ denote the vertex group of $\sG$ at $x$. Up to isomorphism, $\sG_x$ is independent of the choice of $x$. We mention here that we will always compose morphisms in categories and groupoids from left to right.

A functor of groupoids $\mathcal{F}:\tilde{\mathscr{G}} \to \mathscr{G}$  is a \textbf{groupoid covering} if  any morphism in $\mathscr{G}$ has a unique lift in the following sense: given $\alpha \in \Hom_{\mathscr{G}}(x,y)$ and  $\tilde{x} \in \cF^{-1}(x)$, there exist unique $\tilde{y} \in \tilde{\sG}$ and $\tilde{\alpha} \in \Hom_{\tilde{\mathscr{G}}}(\tilde{x}, \tilde{y})$ such that $\mathcal{F}(\tilde{\alpha})=\alpha$. The covering $\cF$ is \textbf{universal} if the vertex group of $\tilde{\sG}$ is trivial. Note that this is equivalent to $\Hom_{\tilde{\sG}}(\tilde{x}, \tilde{y})$ being a singleton for all objects in $\tilde{\sG}$.

The \textbf{deck transformation group} $\Deck(\cF)$ consists of automorphisms $\varphi :\tilde{\sG} \to \tilde{\sG}$ such that $\cF = \cF\circ \varphi$. (Note here that we require actual equality of functors, and not just isomorphisms.)
There is a natural action of $\Deck(\cF)$ on the fibers $\cF^{-1}(x)$, and if this action is transitive (for some, hence any, choice of $x$), we say that $\cF$ is a \textbf{normal cover}. 
In this case, for any $x \in \sG$ and $\tilde{x} \in \cF^{-1}(x)$, we have a morphism $\sG_x \to \Deck(\cF)$ which fits into a short exact sequence
\begin{align}\label{eq:sesfornormalcover}
1\to \tilde{\sG}_{\tilde{x}} \to \sG_x \xrightarrow{\psi} \Deck(\cF) \to 1,
\end{align}
where for each $\alpha \in \sG_x$, the deck transformation $\psi(\alpha)$ sends $\tilde{x}$ to the end point of the lifted path $\tilde{\alpha}$.

An action of a group $N$ on $\sG$ consists of a compatible pair of actions of $N$ on the set of objects and the set of morphisms of $\sG$ (i.e. the action is a map $N \to \Aut(\sG)$ into the group of automorphisms of $\sG$). We let $[x],[\alpha]$, etc.\ denote the equivalence classes of objects and morphisms in $\sG$ with respect to the action of $N$.
Suppose we are given a  group $N$ acting on $\sG$ such that the action is free on the objects. Then we define the \textbf{quotient category} $\sG/N$ as follows. The objects of $\sG/N$ are the equivalences classes of objects of $\sG$ under the action of $N$. The morphisms are given by
\[\Hom_{\sG/N}([x],[y]) \coloneq \Big\{ [\alpha] \;\mid\; \exists g,g' \in N,\; \alpha\in \Hom_{\sG}(g\cdot x,g'\cdot y) \Big\}.\]
Note that since $N$ acts freely on the objects of $\sG$, a morphism $[\alpha]\in \Hom_{\sG/N}([x],[y])$ admits a unique representative in $\sG$ whose source is $x$. 

Composition in $\sG/N$ is given in the following way: Consider $[\alpha],[\beta]$ two morphisms in $\sG/N$ such that the source of $[\beta]$ is the target of $[\alpha]$. Again, there is a unique representative $\beta'$ of $[\beta]$ whose source in $\sG$ is the target of $\alpha$. We can then define $[\alpha]\circ [\beta]$ to be $[\alpha \circ \beta']$. One easily checks that this definition is independent of the choice of representative, and that $\sG/N$ is again a groupoid. Moreover, the natural quotient $\cQ:\sG\to \sG/N$ is a normal covering with $\Deck(\cQ)\cong N$, and we have a short exact sequence
$
1 \to \sG_x \to (\sG/N)_{[x]} \to N \to 1.
$

Let $\cX$ be a (path connected, locally path connected and semilocally simply-connected) space with $N$ acting freely and properly discontinuously.
Then $\cX \to \cX/N$ is a normal (i.e.\ Galois) covering, inducing a groupoid covering on their fundamental groupoids $\pi_1(\cX) \to \pi_1(\cX/N)$.
On the other hand, the action of $N$ also induces an action on the groupoid $\pi_1(\cX)$, to which we can consider the quotient category $\pi_1(\cX)/N$.
It follows from the definition that $\pi_1(\cX/N) = \pi_1(\cX)/N$, and both groupoid coverings $\pi_1(\cX) \to \pi_1(\cX/N)$ and $\pi_1(\cX) \to \pi_1(\cX)/N$ are the same.


Finally, we recall that a \textbf{groupoid completion} of a small category $\sG^+$ is the category $\sG$ obtained by formally inverting all morphisms. The groupoid completion is equipped with a canonical functor $\kappa:\sG^+ \to \sG$ satisfying the following universal property: every  functor from $\sG^+$ to a groupoid factors through $\kappa$.

\section{Tits cone intersection}\label{sec:tits_cone_intersection}
In this section, we recall the definition of Iyama--Wemyss' Tits cone intersection and its associated combinatorics following \cite{IW}.
We also prove that the hyperplane arrangement in its interior is locally-finite, and study the action of the normaliser quotient (see \eqref{def:N(W,J)_and_N(A,J)}) on the Tits cone and its associated complexified hyperplane complement.

Throughout this section, we fix a Coxeter diagram $\Gamma$ with associated Coxeter group $W$. We also fix a subset $J$ of the vertex set $\Gamma_0$ and we otherwise keep the notation from the last section.

\subsection{Basic construction} \label{sec:Jdef}
    To declutter notation, for a set $X$ and an element $a$ we sometimes write:
    \begin{align*}
        X+a &\coloneq X \cup \{a\},\\
        X-a &\coloneq X \setminus \{a\}.
    \end{align*}
    For any subset $I \subseteq \Gamma_0$ we let $I^c\coloneq \Gamma_0 \setminus I$.
    Define the following:
    \begin{align}
        \Theta_I &\coloneqq \{\varphi \in \Theta \mid \varphi(\alpha_i) = 0 \text{ for all } i \in I\}, \\
        C_I &\coloneqq \{ \varphi \in \Theta_I \mid \varphi(\alpha_j) > 0 \text{ for all } j \in I^c \}.
    \end{align}
    The \textbf{Tits cone intersection} of $J$, or the $J$-cone, is defined as 
    \[
    \Cone(J) \coloneqq \Cone \cap \Theta_J,
    \]
    with its interior in $\Theta_J$ (not in $\Theta$) denoted by $\Cone(J)^\circ$.
    The set of $J$\textbf{-roots} is defined as 
    \begin{equation} \label{eqn:J-roots}
    \Phi^J \coloneqq \{ \alpha \in \Phi \mid \Theta_J \not\subseteq H_\alpha \} = \Phi \setminus \Phi_J.
    \end{equation}
    The decomposition $\Phi = \Phi^+ \sqcup \Phi^-$ into positive and negative roots restricts to a decomposition of $\Phi^J$ into positive and negative $J$-roots, which we denote by $\Phi^J_+$ and $\Phi^J_-$ respectively.
    We consider the following set of hyperplanes in $\Theta_J$:
    \[
    \cH^J \coloneqq \{H \cap \Theta_J \mid H \in \cH: \Theta_J \not\subseteq H\} = \{H_\alpha \cap \Theta_J \mid \alpha \in \Phi^J_+ \}.
    \]
    When the context is clear, we will simply refer to the hyperplanes in $\cH^J$ as $H_\alpha$ (instead of $H_\alpha \cap \Theta_J$).
    Note that $\cH^J$ is not necessarily in bijection with the set of positive $J$-roots, as we might have $H_\alpha \cap \Theta_J = H_{\alpha'}\cap \Theta_J$ for distinct positive $J$-roots $\alpha$ and $\alpha'$; see Example \ref{ex:A3.1} and Section \ref{sec:example}.

    Note that $\Cone(J)$ is a convex cone (it is the intersection of two convex cones), and it has the following natural stratification:
    \begin{equation} \label{eqn:JTitsconedecomposition}
        \Cone(J) = \bigsqcup_{L \subseteq \Gamma_0} \bigsqcup_{\substack{x \in W/W_L \\ x \cdot C_L \subseteq \Theta_J}} x \cdot C_L.
    \end{equation}
    In terms of the hyperplane arrangement $(\Cone(J),\cH^J)$, the chambers are $x\cdot C_L \subseteq \Theta_J$ with $|L|=|J|$.
    The faces of each chamber $x\cdot C_L$ are $x\cdot C_{L+i}$ for each $i \in \Gamma_0 - L$.
    Each face $x\cdot C_{L+i}$ is supported on the hyperplane $H_{x\cdot \alpha_i}$, where $\overline{x\cdot C_{L+i}} = H_{x\cdot \alpha_i} \cap \overline{x\cdot C_L}$ is the wall. Notice that we have a similar stratification of $\Cone$ (the entire Tits cone) corresponding to the case $J=\varnothing$.
    
    In general, the hyperplane arrangement $(\Cone(J),\cH^J)$ need not be locally finite -- this is already the case in the classical setting, i.e.\ for an infinite Coxeter group $W$ with $J = \emptyset$ (there are infinitely many hyperplanes and the origin is in $\Cone(\emptyset) = \Cone$).
    Analogous to the classical case, we show that the hyperplane arrangement in its interior: $(\Cone(J)^\circ, \cH^J)$, is a locally-finite hyperplane arrangement.

    To do so we need the following lemma, which is a special case of \cite[Lemma 9]{Vinberg_reflectiongroup}. 
    We included a proof in our setting for the convenience of the reader.
    Note that we say a hyperplane $H$ \textit{separates} two points $f$ and $g$ in $\Cone$ if $f$ and $g$ live on different half spaces, i.e.\ $f$ and $g$ are in different connected components of $\Cone \setminus H$. 
    \begin{Lemma}\label{lem:two_points_separated_finitely_many_hyperplanes}
        Let $f,g \in \Cone$.
        Then the number of hyperplanes in $\cH$ separating $f$ and $g$ is finite.
    \end{Lemma}
    \begin{proof}
        Let $\overline{C} \coloneqq \bigsqcup_{I \subseteq \Gamma_0} C_I$ (the closure of $C=C_\emptyset$ in $\Cone$).
        Since $\overline{C}$ is a fundamental domain for the action of $W$, we may assume without loss of generality that $f \in \overline{C}$ and $g \in w\overline{C}$ for some $w \in W$.
        
        Consider the line segment in $\Cone$ connecting $f$ and $g$ (this is possible since $\Cone$ is convex): 
        \[
        [f,g] \coloneqq \{(1-t)f + tg \mid t\in [0,1] \} \subseteq \Cone.
        \]
        The statement of the lemma is equivalent to having only finitely many hyperplanes intersecting the open line segment $(f,g) \coloneqq [f,g]\setminus\{f,g\}$ without completely containing it (i.e.\ each intersection is a point). 

        We first show that for each $h \in (f,g)$, the set
        \[
        \{H_\alpha \in \cH \mid H_\alpha\cap(f,g) = \{h\}\}
        \]
        is a finite set.
        If $f=g$, then there is nothing to show, so we may assume that $f\neq g$.
        Let $H_{\alpha}$ be a hyperplane containing $h$ with $\alpha \in \Phi^+$ (we can always pick $\alpha$ to be a positive root).
        If $f(\alpha)=0$, then $f$ is also in $H_{\alpha}$ and so $H_{\alpha}$ would contain $(f,g)$.
        Since $f(\alpha)\geq 0$, it follows that 
        \begin{align*}
            \left\{H_\alpha \mid H_\alpha\cap(f,g) =\{h\}\right\}
                &= \{H_\alpha \mid f(\alpha) >0, h(\alpha)=0\} \\
                &= \{H_\alpha \mid g(\alpha)<0, h(\alpha)=0\};
        \end{align*}
        the second equality comes from the fact that $(f,g)$ is an open line segment containing $h$.

        Recall that $g \in w\overline{C}$.
        It follows that $g(\alpha)<0$ if and only if $w^{-1}\cdot \alpha \in \Phi^-$.
        Indeed, $w^{-1}\cdot g \in \overline{C}$ and so $(w^{-1}\cdot g)(w^{-1}\alpha) = g(\alpha) < 0$ if and only if $w^{-1}\cdot\alpha \in \Phi^-$.
        Thus we get that
        \[
        |\{H_\alpha \mid H_\alpha \cap (f,g) = \{h\}\}| \leq |\{ \alpha\in \Phi^+ \mid w^{-1}\cdot \alpha \in \Phi^-\} | = \ell(w^{-1}) < \infty
        \]
        (see e.g.\ \cite[Proposition 4.4.4]{BB_combinatoricsCox} for the equality above).

        We complete the proof by showing that the set
        \[
        \mathfrak{I}_{f,g} \coloneqq \{h \in (f,g) \mid \exists H \in \cH: H\cap[f,g]=\{h\}\}
        \]
        is finite.
        We do this via induction on $\ell(w)$ (where $g \in w\overline{C}$).
        If $\ell(w)=0$, then $f$ and $g$ are both in $\overline{C}$ and it is clear that $\mathfrak{I}_{f,g}$ is empty and hence finite.
        
        We may now assume that $w$ is nontrivial, so that $g\in w\overline{C}\neq \overline{C}$.
        In particular, the following set is non-empty:
        \[
        I \coloneqq \{ i \in \Gamma_0 \mid g(\alpha_i)<0 \}.
        \]
        Define the point $h \in [f,g]$ to be such that $[f,h] = [f,g] \cap \overline{C}$.
        Since $h$ belongs to $\overline{C}$, we have that $h(\alpha_i)\geq 0$ for all $i \in \Gamma_0$.
        We claim that there exists $i \in I$ such that $h(\alpha_i)=0$.
        Indeed, if $h(\alpha_i)>0$ for all $i \in I$, then we would be able to extend $[f,h]$ to some $[f,h'] \supset [f,h]$ such that $[f,h'] \subseteq \overline{C}$, which contradicts the construction of $[f,h]$. 
        
        Let $i \in I$ be such that $h(\alpha_i)=0$.
        Since $g(\alpha_i)<0$ and $g \in w\overline{C}$, we get as before that $w^{-1}\cdot \alpha_i \in \Phi^-$, so that $\ell(s_iw) =\ell(w^{-1}s_i) < \ell (w)$ (see e.g.\ \cite[Proposition 4.4.6]{BB_combinatoricsCox}).
        We can then consider the line segment $s_i\cdot[h,g] = [h,s_i\cdot g]$, with $s_i\cdot g \in (s_iw)\overline{C}$.
        The induction hypothesis applies to both line segments $[h,s_i\cdot g]$ and $[f,h]$.
        Since $|\mathfrak{I}_{h,g}| = |\mathfrak{I}_{h,s_i\cdot g}|$, we conclude that $\mathfrak{I}_{f,g}$ is a finite set.
\end{proof}
    
    \begin{Proposition} \label{prop:locallyfinite}
        Every point $u \in \Cone(J)^\circ$ is contained in an open ball in $\Cone(J)$ that intersects only finitely many hyperplanes in $\cH^J$. In particular, $(\Cone(J)^\circ, \cH^J)$ is a locally-finite hyperplane arrangement.
    \end{Proposition}
    \begin{proof}
        Firstly, note that the interior of a convex space in a finite dimensional vector space is convex, so $\Cone(J)^\circ$ is convex.

        Let $u$ be an element of $\Cone(J)^\circ$. Seeing $u$ as an element of $\Cone$, we can consider a neighbourhood of $u$ in $\Cone$ which intersects only the hyperplanes in $\cH$ containing $u$. For instance, take the union of the strata of the natural stratification of $\Cone$ (from \eqref{eqn:JTitsconedecomposition} with $J=\emptyset$) whose closure contains $u$. We can thus consider a small enough open ball $B$ centred at $u$ so that $B\cap \Cone$ only intersects the hyperplanes in $\cH$ containing $u$. Since $u$ lies in the interior of $\Cone(J)$ in $\Theta_J$, we can also assume (by shrinking $B$ if necessary) that $B\cap \Theta_J$ is an open ball in $\Cone(J)$ that contains $u$.

        We will show that only finitely many hyperplanes in $\cH^J$ intersect $B\cap\Theta_J$ by constructing two points $v$ and $v'$ in $B \cap \Theta(J)$ such that (i) they touch no hyperplanes in $\cH^J$, and (ii) their line segment $[v,v']$ contains $u$ as an interior point. 
        From this, it follows that any $H \in \cH$ satisfying $H \cap \Theta_J \in \cH^J$ and $u \in H$ has to separate $v$ and $v'$ in $\Cone$.
        Since all hyperplanes in $\cH^J$ are of the form $H\cap \Theta_J$, it follows from Lemma \ref{lem:two_points_separated_finitely_many_hyperplanes} that only finitely many hyperplanes in $\cH^J$ intersect $u$ (even though there may be infinitely many hyperplanes in $\cH$ intersecting $u$).
        By the construction of $B$, all hyperplanes in $\cH$ intersecting $B$ must intersect $u$, so we get that $B\cap \Theta_J$ only intersects finitely many hyperplanes in $\cH^J$ as desired.
        
        We now provide the construction of such $v$ and $v'$. 
        By looking at the stratification \eqref{eqn:JTitsconedecomposition}, we see that the complement of all the hyperplanes in $\cH^J$ in $\Cone(J)$ is dense in $\Cone(J)$. 
        This allows us to choose some point $v$ in the open ball $B\cap \Theta_J$ which lies on no hyperplane in $\cH^J$. Being an open ball centred at $u$, $B\cap \Theta_J$ contains the symmetric of $v$ relative to $u$ defined by $v'\coloneq u+(u-v)$. 
        Note that if $v'$ is contained in some hyperplane $H \cap \Theta_J\in \cH^J$, then the hyperplane $H$ intersects $B$ and hence contains $u$ (by construction of $B$). We then have that $H \cap\Theta_J$ also contains $v=u+(u-v')$, which contradicts our choice of $v$. As such, both $v$ and $v'$ are points in $B\cap \Theta_J$ which lie on no hyperplane of $\cH^J$. 
        It is clear from our construction that the segment $[v,v']$ passes through $u$, hence $v$ and $v'$ satisfy both of the desired properties, which completes the proof.
    \end{proof}

    \begin{Remark}\label{rem:locallyfinite}
        While $\Cone(J) \subseteq \Cone$ by definition, note that $\Cone(J)$ (and hence $\Cone(J)^\circ$) can be completely disjoint from $\Cone^\circ$. For example, if $J$ is such that $W_J$ is an infinite parabolic subgroup, then $\Cone(J) \cap \Cone^\circ = \emptyset$.
        In particular, the arrangement $(\Cone(J)^\circ,\cH^J)$ need not be a restriction from $(\Cone^\circ,\cH)$, so the locally-finiteness of $(\Cone(J)^\circ,\cH^J)$ does not immediately follow from the locally-finiteness of $(\Cone^\circ, \cH)$.
        (In fact, our proof of Proposition \ref{prop:locallyfinite} also applies to $(\Cone^\circ, \cH)$ by choosing $J=\emptyset$.)
    \end{Remark}

    \begin{Remark} \label{rmk:comparingtwoarrangements}
    Since chambers are open subsets, it is clear that the set of chambers in $(\Cone(J),\cH^J)$ are the same as the set of chambers in $(\Cone(J)^\circ,\cH^J)$. 
    Moreover, two chambers in $(\Cone(J), \cH^J)$ are adjacent (i.e.\ they share a common wall) if and only if they are adjacent in $(\Cone(J)^\circ,\cH^J)$.
    As such, when referring to chambers and paths between chambers (see Section \ref{sec:geodesicpaths}), we will work with $(\Cone(J)^\circ,\cH^J)$ instead of $(\Cone(J),\cH^J)$. This is easier to do due to the locally-finite property of $(\Cone(J)^\circ,\cH^J)$.
    For example, local-finiteness of $(\Cone(J)^\circ,\cH^J)$ immediately implies that any two chambers in $(\Cone(J)^\circ,\cH^J)$ (hence also in $(\Cone(J),\cH^J)$) are connected by a finite path.
    \end{Remark}

The following Labelling Theorem in \cite{IW} shows us how to algebraically label the chambers of $(\Cone(J)^\circ,\cH^J)$ (which are also the chambers of $(\Cone(J),\cH^J)$ by Remark \ref{rmk:comparingtwoarrangements}).
\begin{Theorem}[\protect{\cite[Theorem 1.15]{IW}}] \label{thm:labellingtheorem}
    Let $\Gamma$ be a Coxeter diagram with associated Coxeter group $W$, and let $J$ be a subset of the vertex set $\Gamma_0$. There is a bijection between the set of chambers of $(\Cone(J)^\circ,\cH^J)$ and the following set
    \begin{equation} \label{eqn:chamberlabels}
    \Cham(J)\coloneq
    \{ (x,I) \mid W_Jx = xW_I \text{ and } x \text{ has minimal length in } xW_I\},
    \end{equation}
    given by sending $(x,I) \mapsto x\cdot C_I$.
\end{Theorem}

\begin{Remark}
    Note that we can have $x\cdot C_I = x'\cdot C_I$ for distinct $x$ and $x'$, and indeed this is the case for any $x' \in xW_I$. The length constraint that appears in \eqref{eqn:chamberlabels}  removes this ambiguity.
\end{Remark}

From here on, we will abuse notation and denote both the set in \eqref{eqn:chamberlabels} and the set of chambers in $(\Cone(J)^\circ,\cH^J)$ by $\Cham(J)$, and accordingly we will sometimes write  $(x,I) \in \Cham(J)$ and other times $x\cdot C_I \in \Cham(J)$ depending on the context. 

\begin{example}\label{ex:A3.1}
    We describe our constructions with a small example. A more involved example is presented in Section \ref{sec:example}. 
    Let $\Gamma$ be of type $A_3$ and let $J=\{3\} \subset \Gamma_0$. Write $\Theta=\bigoplus_{i=1}^3\RR \alpha_i^*$, where $\alpha_i^*$ is the basis of $\Theta=V^*$ which is dual to the standard basis of $V$ given by the simple roots. Then $\Theta_J=\RR\alpha_1^*\oplus \RR\alpha_2^*$, and $\cH^J$ contains the following 3 hyperplanes (lines):
    \begin{align*}
        H_{\alpha_1}\cap \Theta_J  = \RR\alpha_2^*, \;\;
        H_{\alpha_2}\cap \Theta_J  = \RR\alpha_1^*, \;\;
        H_{\alpha_1+\alpha_2}\cap \Theta_J  = \RR(\alpha_1^*-\alpha_2^*).
    \end{align*}
    The arrangement in $\Theta_J = \Cone(J) = \Cone(J)^\circ$ is shown below, where we included the chamber labels given by Theorem \ref{thm:labellingtheorem}:
    \[
    \begin{tikzpicture}[scale=0.9, >=Latex, every node/.style={font=\small}]
        \draw[very thick] (0,-2.35) -- (0,2.35);
        \draw[very thick] (-3.45,0) -- (3.45,0);
        \draw[very thick] (-2.35,2.35) -- (2.35,-2.35);

        \node[above] at (0,2.35) {$\boldsymbol{\alpha}_2^*$};
        \node[right] at (3.45,0) {$\boldsymbol{\alpha}_1^*$};
        \node[right] at (2.35,-2.35) {$\boldsymbol{\alpha}_1^*-\boldsymbol{\alpha}_2^*$};

        \node[font=\scriptsize] at (1.8,1.0) {$(e,3)$};
        \node[font=\scriptsize] at (-0.65,1.7) {$(s_1,3)$};
        \node[font=\scriptsize] at (-2.25,0.8) {$(s_1s_2s_3,2)$};
        \node[font=\scriptsize] at (-1.75,-1.25) {$(s_1s_2s_3s_1s_2,1)$};
        \node[font=\scriptsize] at (1.02,-2.08) {$(s_2s_3s_1s_2,1)$};
        \node[font=\scriptsize] at (2.25,-0.8) {$(s_2s_3,2)$};
    \end{tikzpicture}
    \]
\end{example}


\subsection{Simple wall crossings} 
Let $x\cdot C_I$ be a chamber of the arrangement $(\Cone(J)^\circ,\cH^J)$, i.e. an element of $\Cham(J)$. Recall that a wall of $x\cdot C_I$ is a codimension-one subspace of $\overline{x\cdot C_I}$ given by its intersection with some hyperplane $H \in \cH^J$. 
Note that such a hyperplane can always be described as $H_{x\alpha_a}$ for some $a \in I^c$, and the wall is given by $\overline{x \cdot C_{I+a}}$.
In \cite[Section 1.5]{IW}, Iyama--Wemyss provide us explicit formulae that relate two adjacent chambers sharing a common wall.
To explain this  we need to introduce some notation. 

\begin{Notation}
Let $I \subseteq \Gamma_0$ and let $a \in I^c$. 
Following \cite[Notation 1.31]{IW}, let $\Gamma(I+a)$ denote the induced subgraph of $\Gamma$ on the vertices $I+a$. We define $[I,a]$ to be the set of vertices of the connected component of $\Gamma(I+a)$ that contains $a$.
\end{Notation}

Now let $(x,I) \in \Cham(J)$ and suppose that we have $a \in I^c$ such that $W_{[I,a]}$ is a finite parabolic subgroup. We set 
\begin{equation} \label{eqn:BHgen}
v_{a,I} \coloneqq w_{[I,a]-a}w_{[I,a]} \in W,    
\end{equation} 
where $w_K$ denotes the longest element in the finite parabolic subgroup $W_K \subseteq W$ (see Notation \ref{notation:longest_element}).
Set $K \coloneqq I+a-\iota_{[I,a]}(a)$ (see \eqref{eqn:diagraminvolution} for definition of $\iota_{[I,a]})$. The \textbf{simple wall crossing} at $a$ is given by
\begin{equation} \label{eqn:wallcrossingformula}
\omega_{a,I}(x,I) \coloneqq (xv_{a,I}, K).
\end{equation}
Note that $v_{a,I}$ can be alternatively characterised as the unique element for which $v_{a,I}(\Phi_{I+a}^+\setminus \Phi_K^+) \subseteq \Phi_K^-$ and $v_{a,I}(\Phi_K^+) \subseteq \Phi_{I+a}^+$ (cf.\ paragraph after Proposition 2.1 in \cite{BH}).

\begin{Remark}\label{rem:v_{a,varnothing}}
    If $J$ is  empty then it is its only associate. For all $a\in \Gamma_0=J^c$, we have $[\varnothing,a]=\{a\}$ and thus $W_{[\varnothing,a]}=W_a$ is a cyclic group of order $2$. The element $v_{a,\varnothing}$ is then the simple reflection $s_a$. 
\end{Remark}

The terminology of simple wall crossing is justified by the following:

\begin{Proposition}[\protect{\cite[Theorem 1.32]{IW}}]\label{prop:simplewallcrossing}
    Let $(x,I), (y,K) \in \Cham(J)$ be two adjacent chambers.
    Then there exists a unique $a \in I^c$ and a unique $a' \in K^c$ such that $I+a = K+a'$. Furthermore,
    \begin{enumerate}
        \item the common wall is given by
        \[
        \overline{x\cdot C_{I+a}}= \overline{y\cdot C_{K+a'}},
        \]
        \item $W_{[I,a]} = W_{[K,a']}$ is a finite parabolic subgroup, and
        \item the two chambers are related by the simple wall crossings formula \eqref{eqn:wallcrossingformula}:
    \begin{align*}
        \omega_{a,I}(x,I)  &= (y,K), \\
        \omega_{a',K}(y,K) &= (x,I).
    \end{align*}
    \end{enumerate}
\end{Proposition}
Note also that $v_{a,I}v_{a',K} = v_{a',K}v_{a,I} = e$.
\begin{Remark}\label{rem:wallcrossingshorthand}
In \cite[Definition 1.28]{IW} the simple wall crossing is denoted $\omega_a$. 
In situations where no confusion can arise, we will also adopt this notation, but in general we record also $I$ to avoid ambiguity. 
\end{Remark}

\begin{Remark} \label{rmk:oppositev}
    The elements $v_{a,I}$ defined here are the inverse of $v[a,I]$ in the notation of \cite{BH}. 
    In fact, Brink and Howlett define more generally elements $v[I,J]$ for certain disjoint subsets $I$ and $J$ \cite[Section 2]{BH}.
\end{Remark}

We recall some crucial results from \cite{IW}.
\begin{Lemma} \label{lem:combinatoriallemmas}
    \begin{enumerate}
        \item \label{item:Jrootsfiniteness}If $|\Phi_{I+a}-\Phi_I| < \infty$, then $W_{[I,a]}$ is finite.
        \item \label{item:lengthandhalfspace} Let $(x,I)$ and $\omega_a(x,I)$ be adjacent chambers. Then $\ell(xv_{a,I}) = \ell(x)+\ell(v_{a,I})$ if and only if $(e,J)$ and $(x,I)$ are on the same half-space separated by $H_{x\cdot\alpha_a}$.
        \item \label{item:chambersfinitewallcrossing} Any two chambers are related by a finite sequence of wall-crossings.
        \item \label{item:reducedwallcrossing} If $W_{I+a}$ is finite (so $W_{[I,a]}$ is also finite), then the simple wall crossing formula \eqref{eqn:wallcrossingformula} reduces to
        \begin{equation} \label{eqn:reducedwallcrossingformula}
            \omega_{a,I}(x,I) = (xw_Iw_{I+a}, I+a-\iota_{I+a}(a))
        \end{equation}
    \end{enumerate}
\end{Lemma}
\begin{proof}
    The first statement is \cite[Lemma 1.40]{IW}.
    The second and third statements are Proposition 1.42 and 1.43 in \cite{IW} respectively.
    For the final statement, suppose that the induced subgraph $\Gamma(I+a)$ is a product of disjoint connected components $\Gamma(I_1)\times\cdots\times \Gamma(I_r)$ and suppose that $a \in I_1$. Then $I_1=[I,a]$ and we have 
    \begin{align*}
        w_Iw_{I+a}=(w_{I_1-a}w_{I_2}\cdots w_{I_r})(w_{I_1}w_{I_2}\cdots w_{I_r}) = w_{I_1-a}w_{I_1}=v_{a,I}.
    \end{align*}
    Similarly, $\iota_{I+a}(a)=\iota_{I_1}(a)=\iota_{[I,a]}(a)$.
\end{proof}

\begin{Remark}
    Statement \eqref{item:chambersfinitewallcrossing} in Lemma \ref{lem:combinatoriallemmas} is also an immediate consequence of the local-finiteness of the hyperplane arrangement $(\Cone(J)^\circ,\cH^J)$ (whose chambers are also chambers of $(\Cone(J),\cH^J)$; see Remark \ref{rmk:comparingtwoarrangements}).
\end{Remark}

\subsection{Normaliser actions and complexified hyperplane complements}\label{sec:normaliser_action} In this section, we study the action of the normaliser $\mathrm{Norm}_W(W_J)$ of the standard parabolic subgroup $W_J$ on the space $\Theta_J$, the Tits cone intersection $\Cone(J)$ and the set of chambers $\Cham(J)$, in order to establish the short exact sequence \eqref{eq:ses3} mentioned in the introduction. 

We begin by describing this action and a few of its properties.

\begin{Lemma}\label{lem:NWJaction}
    The action of $\mathrm{Norm}_W(W_J) \subseteq W$ on $\Theta$ preserves $\Theta_J$. The kernel of the action of $\mathrm{Norm}_W(W_J)$ on $\Theta_J$ is precisely $W_J$, so that we obtain a faithful linear action of $N(W,J)\coloneq \mathrm{Norm}_W(W_J)/W_J$ on $\Theta_J$.
    This action of $N(W,J)$ further restricts to an action on $\Cone(J)$ that is compatible with the stratification in \eqref{eqn:JTitsconedecomposition}: if $\overline{g}\in N(W,J)$, then we have $\overline{g}\cdot(x\cdot C_L)=(gx)\cdot C_L$ where $g$ is any lift of $\overline{g}$ in $\mathrm{Norm}_W(W_J)$.
\end{Lemma}
\begin{proof}
    Firstly, note that $g \in \mathrm{Norm}_W(W_J)$ implies that $g^{-1}s_jg$ is a reflection in $W_J$ for all $j \in J$.
    As such, for all simple root $\alpha_j$ with $j \in J$, we get $g^{-1}\cdot \alpha_j \in \Phi_J$, so that $g^{-1}\cdot \alpha_j = \sum_{j' \in J}c_{j'}\alpha_{j'}$ is a linear combination that involves only the simple roots from $J$.
    Thus for all $\varphi \in \Theta_J$, we have $(g\cdot \varphi)(\alpha_j)=\varphi(g^{-1}\cdot\alpha_j) = 0$ for all $j \in J$, which shows that $g\cdot \varphi \in \Theta_J$.

    Under the $W$-action on $\Theta$, the stabiliser of any point in $C_J \subseteq \Theta_J \subseteq \Theta$ is exactly $W_J$ (see for instance \cite[Chapter V. \S 4, Proposition 5]{Bourbaki_4-6}).
    Since $W_J \subseteq \mathrm{Norm}_W(W_J)$, it follows that the kernel of the action of $\mathrm{Norm}_W(W_J)$ on $\Theta_J$ is also $W_J$.
    The linearity of the action of $N(W,J) = \mathrm{Norm}_W(W_J)/W_J$ is an immediate consequence of the linearity of the action of $W \supseteq \mathrm{Norm}_W(W_J)$.

    The final statement follows immediately from the stratification of $\Cone(J)$ in \eqref{eqn:JTitsconedecomposition}.
    Indeed, given $g\in \mathrm{Norm}_W(W_J)$ and a stratum $x\cdot C_L \subseteq \Cone(J) \subseteq \Theta_J$, the equality $g \cdot (x\cdot C_L) = (gx)\cdot C_L$ is immediate from the definition of the action and we have that $g \cdot (x\cdot C_L) \subseteq \Theta_J$ since $g\cdot \Theta_J \subseteq \Theta_J$.
\end{proof}

\begin{Lemma}\label{lem:normaliseraction}
Recall the action of $N(W,J)$ on $\Cone(J)$ from Lemma \ref{lem:NWJaction}.
    \begin{enumerate}
        \item \label{item:actionminlength} Each $\overline{g} \in N(W,J)$ acts on $\Cham(J)$ by $\overline{g}\cdot (x,I)=(y,I)$, where $y$ is the minimal length representative of the coset $gxW_I$ for $g \in \mathrm{Norm}_W(W_J)$ any representative of $\overline{g}$.
        \item \label{item:transitive1stlabel} The action of $N(W,J)$ on $\Cham(J)$ is transitive on the first label in the sense that any two chambers of the form $(x,I)$ and $(y,I)$ lie in the same $N(W,J)$-orbit. 
        \item \label{item:freeonchambers} The action of $N(W,J)$ on $\Cham(J)$ is free. 
        \item \label{item:preserveadjacency} The action of $N(W,J)$ on $\Cham(J)$ preserves adjacency: precisely if $\omega_a(C)=C'$ with $C,C'\in \Cham(J)$, then $\omega_a(g\cdot C)=g \cdot C'$ for any $g\in N(W,J)$.
    \end{enumerate}
\end{Lemma}
\begin{proof}
    
    \eqref{item:actionminlength} This is a direct application of Theorem \ref{thm:labellingtheorem} and Lemma \ref{lem:NWJaction}. Indeed, each label $(x,I)$ corresponds to the chamber $x \cdot C_I$, and thus $\overline{g}\cdot (x,I)$ must be the label of the chamber $(gx)\cdot C_I$. Under the bijection in Theorem \ref{thm:labellingtheorem}, this label is given by $(y,I)$ where $y$ is the minimal length representative of $gxW_I$.
    
    \eqref{item:transitive1stlabel} Let $(x,I)$ and $(y,I)$ be two chambers in $\Cham(J)$. By Theorem \ref{thm:labellingtheorem}, we have $W_Jx=xW_I$ and $W_Jy=yW_I$, which implies
    \[W_Jyx^{-1}=yW_Ix^{-1}=yx^{-1}W_J.\]
    In other words, $yx^{-1}$ belongs to $\mathrm{Norm}_W(W_J)$. Its image in $N(W,J)$ then maps $(x,I)$ to $(y,I)$ by definition of the action, which proves transitivity.

    \eqref{item:freeonchambers} For $(x,I) \in \Cham(J)$, let $\overline{g}\in N(W,J)$ be such that $\overline{g}\cdot (x,I)=(x,I)$ and let $g \in \mathrm{Norm}_W(W_J)$ be any representative of $\overline{g}$.
    By \eqref{item:actionminlength} above, $\overline{g}\cdot (x,I)=(x,I)$ implies that $gxW_I=xW_I$, thus
    \[
    W_Jx = xW_I = gxW_I=gW_Jx,
    \]
    showing that $g \in W_J$ as required.

    \eqref{item:preserveadjacency} This is an immediate consequence of the linearity of the action of $N(W,J)$ on $\Theta_J$ (Lemma \ref{lem:NWJaction}). Moreover, if two chambers $x\cdot C_I$ and $y\cdot C_K$ shares a common wall supported on the hyperplane $H_{x\cdot\alpha_a}$ (so that $\omega_{a,I}(x,I) = (y,K)$), then $gx\cdot C_I$ and $gy\cdot C_K$ shares a common wall supported on the hyperplane $H_{gx\cdot\alpha_a}$.
    Hence the algebraic description also follows.
\end{proof}

\begin{Remark}
    Unlike the $J=\emptyset$ case, the action of $N(W,J)$ need not be transitive on the whole set of chambers; see Example \ref{ex:A3.1} and Section \ref{sec:example}.
\end{Remark}

Recall from Proposition \ref{prop:locallyfinite} that the hyperplane arrangement $(\Cone(J)^\circ, \cH^J)$ in the interior of $\Cone(J)$ is a locally-finite hyperplane arrangement. 
We abuse terminology again just like in Notation \ref{not:complexified_complement} and refer to the space 
\begin{equation}\label{eqn:complexifiedJcone}
\cX_J^\CC \coloneqq \Big(\Cone(J)^\circ\times\Cone(J)^\circ\Big) \setminus \bigcup_{H \in \cH^J} H\times H
\end{equation}
as the \textbf{complexified hyperplane complement} of the arrangement $(\Cone(J)^\circ, \cH^J)$.
Note that when $\Gamma$ is spherical, we have that $\Cone(J)^\circ=\Theta_J$, and $$\cX_J^\CC \cong \left( \Theta_J\otimes_\RR\CC \right) \setminus \bigcup_{H \in \cH^J}H\otimes_\RR\CC.$$ 
\begin{Remark} \label{rmk:salvetticomplex}
    Since $(\Cone(J)^\circ, \cH^J)$ is a locally-finite arrangement in an open convex cone, $\cX_J^\CC$ has a finite CW-model given by the so-called Salvetti complex in \cite[Section 3.1]{Paris_Kpi1}.
\end{Remark}

Since $N(W,J)$ acts on $\Cone(J)$, and hence on $\Cone(J)^\circ$, we have a well-defined action of $N(W,J)$ on $\cX_J^\CC$. Since the action of $N(W,J)$ on chambers is free by Lemma \ref{lem:normaliseraction}\eqref{item:freeonchambers}, the action of $N(W,J)$ on $\cX_J^\CC$ is both free and properly discontinuous.
Thus the associated quotient map $\cX_J^\CC \to \cX_J^\CC/N(W,J)$ is a normal (i.e.\ Galois) covering of topological spaces, which induces the following short exact sequence, generalising the one in \eqref{eqn:classicalSES} (where $x \in \cX_J^\CC$):
\begin{align}\label{eq:TitsconeSES}
1 \to \pi_1(\cX_J^\CC,x) \to \pi_1(\cX_J^\CC/N(W,J),x) \to N(W,J) \to 1.  
\end{align}

\subsection{Geodesics paths} \label{sec:geodesicpaths} We recall some familiar notions from the study of locally-finite hyperplane arrangements.
    A \textbf{path} $p$ in the hyperplane arrangement $(\Cone(J)^\circ,\cH^J)$ is a sequence of chambers $$p=\big((x_1,I_1),(x_2,I_2),...,(x_m,I_m)\big)$$ such that for each $i$,  $(x_i, I_i)$ and $(x_{i+1},I_{i+1})$ are adjacent. The \textbf{source} of $p$ is $s(p)= (x_1,I_1)$, the \textbf{target} of $p$ is $t(p)=(x_m,I_m)$, and the \textbf{length} of $p$ is $m-1$. 
    
    If $p$ and $q$ are two paths such that $t(p)=s(q)$, we let $p*q$ denote their concatenation, which is a path from $s(p)$ to $t(q)$. 
    If $(x,I)$ and $(y,K)$ are adjacent chambers related by $\omega_{a,I}$, we let $\omega_{a,I}:(x,I)\to(y,K)$ denote the length one path from $(x,I)$ to $(y,K)$ which we also call a \textbf{simple wall crossing} (cf. part (3) of Proposition \ref{prop:simplewallcrossing}). As in Remark \ref{rem:wallcrossingshorthand}, if the source is clear we will simply denote the simple wall crossing by $\omega_a$. 

To a path $p$ we can associate by Proposition \ref{prop:simplewallcrossing} a sequence of simple wall crossings satisfying $\omega_{a_i,I_i}(x_i,I_i) = (x_{i+1},I_{i+1})$, i.e. $$p=\omega_{a_1,I_1}*\cdots*\omega_{a_{m-1},I_{m-1}}.$$ 
We define 
$$
v_p \coloneq v_{a_1,I_1}\cdots v_{a_{m-1},I_{m-1}} \in W
$$

Finally, $p$ is a \textbf{geodesic path} if it crosses a minimal number of walls between the chambers $(x_1,I_1)$ and $(x_m,I_m)$. Equivalently, $p$ crosses any hyperplane in $\cH^J$ at most once.

\begin{Proposition} \label{prop:geodesicsandwordlength}
    A path $p=\big((x_1,I_1),(x_2,I_2),...,(x_m,I_m)\big)$ is a  geodesic if and only if $$\ell(v_p) = \sum_{i=1}^{m-1}\ell(v_{a_i,I_i}).$$
\end{Proposition}
\begin{proof}
    First the consider the case where $p$ is a path starting at $(e,J)$.
    We induct on the length of $p$.
    The statement is trivial if $p$ has length $1$.
    Now assume for induction that the statement is true for length $n-1$, and consider a path $p$ of length $n$ with $(x,I)\to \omega_a(x,I)$ being the last step of $p$, i.e. $p=q*\omega_a$ for some path $q$ of length $n-1$.
    Then $p$ is a geodesic if and only if $q$ is a geodesic that does not cross the hyperplane $H_{x\cdot \alpha_a}$.
    Therefore $p$ is a geodesic if and only if $q$ is a geodesic and $(x,I)$ and $(e,J)$ are on the same half space separated by $H_{x\cdot \alpha_i}$.
    The statement now follows from Lemma \ref{lem:combinatoriallemmas}\eqref{item:lengthandhalfspace}.

    Now consider instead that we have a path starting at an arbitrary chamber $(z,L)$.
    Then $z^{-1} \cdot \Theta_J = \Theta_L$, and this induces a bijection from $\Cham(J)$ to $\Cham(L)$. This bijection preserves and reflects geodesic paths.
    Moreover, the simple wall crossing $\omega_a: (x,I) \to (xv_{a,I},I+a-\iota_{[I+a]}(a))$ is sent to the simple wall crossing $\omega_a: (x',I) \to (x'v_{a,I},I+a-\iota_{[I+a]}(a))$, with $x'$ the minimal length representative in $z^{-1}xW_I$.
    It follows that any path $p$ is sent to a path $z^{-1}\cdot p$ with the property that $v_{z^{-1}\cdot p} = v_p \in W$.
    In particular, whether or not $v_p$ is length additive with respect to each simple wall crossing is invariant under the bijection.
    Thus we can apply $z^{-1}$ so that our path starts at $(e,L)$ and the result follows from the argument above.
\end{proof}

\subsection{The Deligne groupoid}\label{sec:Dgroupoid}
    The Deligne groupoid was introduced in \cite{Del72} to study the topology of finite simplicial arrangements (and prove the $K(\pi,1)$ conjecture for finite real reflection groups). Although our setting is more general than that of Deligne, the construction of the Deligne groupoid can be carried out similarly. However, the key features proven by Deligne, namely that the Deligne groupoid is a Garside groupoid equivalent to the fundamental groupoid of the complexified complement, only hold (to our knowledge) in the setting of finite simplicial arrangements.
    
    To the locally-finite arrangement $(\Cone(J)^\circ,\cH^J)$, we associate an oriented graph on the set of vertices $\Cham(J)$. To every pair of adjacent chambers we associate  a 2-cycle, i.e.\ given two adjacent chambers $(x,I)$ and $(y,K)$ there is one arrow $(x,I) \to (y,K)$ and another opposite arrow $(y,K) \to (x,I)$.
    We continue to use the notation $\omega_{a,I}$ as above to label these arrows.
    
    A path in this graph from $(x,I)$ to $(y,K)$ is a \textbf{geodesic} if it is a path of minimal length between $(x,I)$ and $(y,K)$.
    Equivalently, it crosses every hyperplane $H \in \cH^J$ at most once.
    
    Define $\D^+ \coloneqq \D^+(J)$ as the category with:
    \begin{itemize}
        \item objects given by $\Cham(J)$; and
        \item morphisms $\Hom_{\D^+}((x,I),(y,K))$ are paths from $(x,I)$ to $(y,K)$ modulo the relation generated by the identification of any two geodesic paths with the same source and target.
    \end{itemize}
    The arrangement groupoid $\D \coloneqq \D(J)$ associated to $(\Cone(J)^\circ,\cH^J)$ is defined as the groupoid completion of $\D^+$.
If $\Gamma$ is of spherical-type, then $(\Cone(J)^\circ,\cH^J) = (\Theta_J,\cH^J)$ is a finite simplicial arrangement, and we call $\sD$ the \textbf{Deligne groupoid}. 
In this case, $\sD$ is equivalent to the fundamental groupoid of $\cX^\CC_J$, denoted $\pi_1(\cX^\CC_J)$ \cite{Del72}.


\begin{Remark} \label{rmk:NormactiononD}
Note that the action of $N(W,J)$ on $\Cham(J)$ from Lemma \ref{lem:normaliseraction} induces an action of $N(W,J)$ on $\D$. Indeed, an element $g \in N(W,J)$ acts on the simple wall crossing $\omega_{a,I}:(x,I) \to (y,K)$  by mapping it to the simple wall crossing $\omega_{a,I}:g\cdot(x,I) \to g\cdot(y,K)$ (see Lemma \ref{lem:normaliseraction}\eqref{item:preserveadjacency}). This action clearly preserves geodesics, since if a path $p$ crosses some hyperplane more than once  the same is also true for the path $g\cdot p$. Therefore we have a well-defined action of $N(W,J)$ on $\D$. 
\end{Remark}

    \begin{example}\label{ex:A3.DeligneGraph}
    Continuing Example \ref{ex:A3.1}, we write $\omega_{a,i}$ for $\omega_{a,\{i\}}$. The oriented graph has one vertex for each of the six chambers and a pair of opposite arrows across each wall:
    \vspace{1mm}
    \[
	    \begin{tikzpicture}[
	        chamber/.style={font=\scriptsize, inner sep=1pt},
	        arrow/.style={->, >=Latex, shorten >=3pt, shorten <=3pt},
	        arrowlabel/.style={font=\scriptsize, inner sep=1pt}
	    ]
	        \node[chamber] (A) at (2.55,1.25) {$(e,3)$};
	        \node[chamber] (B) at (0,2.35) {$(s_1,3)$};
	        \node[chamber] (C) at (-3.05,1.25) {$(s_1s_2s_3,2)$};
        \node[chamber] (D) at (-3.05,-1.25) {$(s_1s_2s_3s_1s_2,1)$};
	        \node[chamber] (E) at (0,-2.35) {$(s_2s_3s_1s_2,1)$};
	        \node[chamber] (F) at (2.55,-1.25) {$(s_2s_3,2)$};

	        \draw[arrow, bend left=16] (A) to (B);
	        \draw[arrow, bend left=16] (B) to (A);
	        \node[arrowlabel] at (1.35,1.28) {$\omega_{1,3}$};
	        \node[arrowlabel] at (1.45,2.3) {$\omega_{1,3}$};

	        \draw[arrow, bend left=16] (B) to (C);
	        \draw[arrow, bend left=16] (C) to (B);
	        \node[arrowlabel] at (-1.5,1.28) {$\omega_{2,3}$};
	        \node[arrowlabel] at (-1.85,2.3) {$\omega_{3,2}$};

	        \draw[arrow, bend left=16] (C) to (D);
	        \draw[arrow, bend left=16] (D) to (C);
	        \node[arrowlabel] at (-2.22,0.22) {$\omega_{1,2}$};
	        \node[arrowlabel] at (-3.7,-0.22) {$\omega_{2,1}$};

	        \draw[arrow, bend left=16] (D) to (E);
	        \draw[arrow, bend left=16] (E) to (D);
	        \node[arrowlabel] at (-1.15,-1.22) {$\omega_{3,1}$};
	        \node[arrowlabel] at (-1.9,-2.4) {$\omega_{3,1}$};

	        \draw[arrow, bend left=16] (E) to (F);
	        \draw[arrow, bend left=16] (F) to (E);
	        \node[arrowlabel] at (0.95,-1.22) {$\omega_{2,1}$};
	        \node[arrowlabel] at (1.65,-2.4) {$\omega_{1,2}$};

	        \draw[arrow, bend left=16] (F) to (A);
	        \draw[arrow, bend left=16] (A) to (F);
	        \node[arrowlabel] at (1.9,0.22) {$\omega_{3,2}$};
	        \node[arrowlabel] at (3.38,-0.22) {$\omega_{2,3}$};
	    \end{tikzpicture}
    \]
\end{example}

\section{The geometry of the Brink--Howlett groupoid} \label{sec:BHgroupoid}
In \cite{BH} Brink and Howlett introduced a groupoid in order to study the normaliser of a standard parabolic subgroup in a Coxeter group. In this section, we give an explicit realisation of its universal cover in terms of the geometry of the associated Tits cone intersection. In Corollary \ref{cor:BHthmA} we use this to give a new proof to parts of one of Brink and Howlett's main results (Theorem A in \cite{BH}). We then explain in Section \ref{sec:atomicM} how to use the connection between Brink and Howlett's work and the Tits cone intersection to answer questions of Iyama and Wemyss about ``atomic Matsumoto relations''. 

Throughout this section, we fix a Coxeter diagram $\Gamma$ with associated Coxeter group $W$. We also fix a subset $J$ of the vertex set $\Gamma_0$. We otherwise keep the notation of Section \ref{sec:tits_cone_intersection}.

\subsection{The Brink--Howlett groupoid and its universal cover} 
    Recall that the reflection representation $V$ of $W$ has basis given by the simple roots $\alpha_k$ for $k \in \Gamma_0$. The \textbf{Brink--Howlett groupoid} $\BH \coloneqq \BH(J)$ is defined as follows:
    \begin{itemize}
        \item The objects are subsets $I \subseteq \Gamma_0$ which are associates of $J$ (see Definition \ref{def:associate}).
        \item The morphism space $\Hom_{\BH}(I,K)$ consists of  $x \in W$ such that 
        $$
        \{\alpha_i \mid i \in I\} =  \{ x\cdot\alpha_k \mid k \in K \}.$$ 
        Composition is given by multiplication in $W$ on the right.
    \end{itemize}

\begin{Remark}
    Note that $\BH$ is the opposite category of the groupoid in \cite{BH}.
    Namely, our composition convention follows path composition which reads from left to right: if $I,K,L$ are associates of $J$, then composition in $\BH$ is given by:
    \begin{align*}
        \Hom_\BH(I,K) \times \Hom_\BH(K,L) &\to \Hom_\BH(I,L) \\
         (x,y) &\mapsto xy.
    \end{align*}
    This is chosen so that it matches the geometric description  of the universal cover of $\BH$ below.
\end{Remark}

\begin{Lemma}
    Suppose $I \in \BH$ and $a \in I^c$ such that $W_{[I,a]}$ is finite.
    Then $$v_{a,I} \in \Hom_\BH(I,I+a-\iota_{[I,a]}(a)),$$
    with $v_{a,I}$ defined as in \eqref{eqn:BHgen}.
\end{Lemma}

\begin{proof}
    Let $k$ be an arbitrary element in $K \coloneqq I+a-\iota_{[I,a]}(a)$. Since $I,K$ are both finite and since $|I|=|K|$, it is sufficient to show that there exists some $i$ in $I$ such that $v_{a,I}(\alpha_k)= \alpha_i$. If $k$ belongs to $K \setminus[I,a]$, then $k$ also belongs to $I$ since $k \neq a$. Moreover, we then have $v_{a,I}(\alpha_k)=\alpha_k$ since $w_{[I,a]}(\alpha_k)=w_{[I,a]-a}(\alpha_k)=\alpha_k$. 
    
    On the other hand, if $k$ belongs to $K\cap [I,a]$, then $k\neq \iota_{[I,a]}(a)$ implies that $\iota_{[I,a]}(k)$ belongs to $[I,a]-a \subseteq I$. By denoting $i \coloneqq \iota_{[I,a]-a}(\iota_{[I,a]}(k)) \in I$, we get
    $$
    v_{a,I}(\alpha_k)=-w_{[I,a]-a}(\alpha_{\iota_{[I,a]}(k)}) = \alpha_i
    $$
    as required.
\end{proof}

\begin{Definition}\label{defn:universalgroupoid}
    Define the groupoid  $\wt{\BH} \coloneqq \wt{\BH}(J)$ as follows.
    \begin{itemize}
        \item The set of objects is $\Cham(J)$.
        \item The morphism space $\Hom_{\wt{\BH}}((x,I),(y,K))$ is the singleton consisting of $x^{-1}y \in W$. Composition is induced by multiplication in $W$ (on the right).
    \end{itemize}
\end{Definition}

Note that if $(x,I)$ and $(y,K)$ are adjacent chambers, by Proposition \ref{prop:simplewallcrossing}, they are related by simple wall crossings: $\omega_{a,I}(x,I) = (xv_{a,I},K)=(y,K)$, and so $\Hom_{\wt{\BH}}((x,I),(y,K)) = \{ v_{a,I} \}$.
More generally, the following result shows that the (unique) morphism between any two chambers can be obtained from a (non-unique) path between the chambers in the arrangement $(\Cone(J)^\circ,\cH^J)$, or equivalently, a sequence of simple wall crossings relating them.
\begin{Proposition}\label{prop:genWC}
    Given two chambers $(x,I), (y,K) \in \Cham(J)$, there exists a (not necessarily unique) path 
    \[
    p = \omega_{a_1,I_1}*\omega_{a_2,I_2}*\cdots*\omega_{a_m,I_m}: (x,I) \to (y,K)
    \]
    in $(\Cone(J)^\circ,\cH^J)$ so that $\Hom_{\widetilde{\BH}}((x,I),(y,K)) = \{v_p\}$, where $v_p = v_{a_1,I_1}v_{a_2,I_2}\dots v_{a_m,I_m} \in W$.
    In other words, the groupoid $\wt{\BH}$ is generated by morphisms $v_{a,I}$ between adjacent chambers.
\end{Proposition}
\begin{proof}
    Since the space of morphisms between any two object is a singleton, this is an immediate consequence of the fact that any two chambers are related by a finite sequence of simple wall crossings (Lemma \ref{lem:combinatoriallemmas}\eqref{item:chambersfinitewallcrossing}).
\end{proof}

\begin{Remark}\label{rem:obvious_functor}
    There is an obvious functor $\D^+ \to \wt{\BH}$ (hence also from $\D$) which is the identity on objects and sends a morphism $p:(x,I) \to (y,K)$ in $\D^+$, which is simply a path in $(\Cone(J)^\circ,\cH^J)$, to the morphism $v_p \in \Hom_{\wt{\BH}}((x,I),(y,K))$ as defined in Proposition \ref{prop:genWC}. 
    By construction, this functor is full, but it is almost never faithful (except in degenerate cases).
    For example, all paths that start and end at the same chamber $(x,I)$ in $\D^+$, trivial or not, are sent the identity morphism in $\Hom_{\wt{\BH}}((x,I),(x,I))$.
\end{Remark}

\begin{Proposition} \label{prop:hominWCandBH}Let $I,K$ be associates of $J$.
    \begin{enumerate}
        \item \label{item:WCtoBH} If $(x,I)$ and $(y,K)$ are chambers in $(\Cone(J)^\circ,\cH^J)$, then the element $x^{-1}y\in W$ lies in $\Hom_\BH(I,K)$.
        \item \label{item:uniquelift} Conversely, if $w$ lies in $\Hom_\BH(I,K)$, then for any chamber $(x,I)$ in $(\Cone(J)^\circ,\cH^J)$, we have $(xw,K) \in \Cham(J)$ so that $w \in \Hom_{\wt{\BH}}((x,I), (xw,K))$.
    \end{enumerate}
\end{Proposition}
\begin{proof}
    For \eqref{item:WCtoBH}, it suffices to show for all $(y,K) \in \Cham(J)$, we have $y \in \Hom_\BH(J,K)$. Point \eqref{item:WCtoBH} then follows via composition (and taking inverses).
    
    By the definition of the labeling of chambers, we have
    \[
    yC_K\subseteq \Theta_J,
    \]
    where $|J| = |K|$ and moreover $y$ has minimal length in $W_Jy = yW_K$.
    Thus for all $\varphi \in C_K$ and all $j \in J$, $(y\cdot\varphi)(\alpha_j) = \varphi(y^{-1}\cdot \alpha_j) = 0$.
    Since $y^{-1}\cdot \alpha_j$ is a root, which is either positive or negative, $y^{-1}\cdot \alpha_j$ must be a linear combination of $\alpha_k$ for $k \in K$.
    In particular, $y^{-1}\cdot \alpha_j$ is a root of $W_K$ for any $j \in J$.
    
    We claim that for any $j \in J$, $y^{-1}\cdot \alpha_j$ is a positive root.
    Indeed, since $y$ has minimal length in $W_Jy$,  $y^{-1}$ has minimal length in $y^{-1}W_J$.
    If $y^{-1}\cdot \alpha_j$ is a negative root, then $\ell(ys_j) < \ell(y)$, contradicting length-minimality.
    The same argument using length-minimality in $yW_K$ instead shows that $y\cdot \alpha_k$ are also all positive roots of the parabolic $W_J$.
    Now $y$ and $y^{-1}$ are linear isomorphisms between the vector spaces $\Span_\RR\{\alpha_j\}_{j\in J}$ and $\Span_\RR\{ \alpha_k\}_{k \in K}$, whose matrices in terms of these fixed bases are both non-negative matrices.
    This implies that the matrix of $y$ (and $y^{-1}$) is a permutation of a diagonal matrix with non-negative entries.
    It follows that the matrix of $y$ must be a permutation matrix, as non-identity positive scalar multiple of a simple root is not a root.
    As such, $y\cdot \{\alpha_k \mid k \in K\} = \{\alpha_j \mid j \in J\}$, which says that $y \in \Hom_\BH(J,K)$ by definition. 
    This completes the proof of \eqref{item:WCtoBH}

    For \eqref{item:uniquelift}, let $w \in \Hom_\BH(I,K)$. Then for all $i \in I$, there exists $k \in K$ such that $w^{-1}\cdot \alpha_i = \alpha_k$.
    It follows that $w \cdot C_K \subseteq \Theta_I$.
    Since $(x,I) \in \Cham(J)$,  we get $x \cdot \Theta_I = \Theta_J$,
    which implies that $xw\cdot C_K \subseteq \Theta_J$.
    We claim that $xw$ is minimal in $xwW_K$.
    Indeed, if $xw$ is not minimal, then there is some $k\in K$ such that $xw(\alpha_k) = x(\alpha_i) \in \Phi^-$, which contradicts the length-minimality of $x$ in $xW_I$.
\end{proof}

\begin{Theorem}\label{thm:WCuniversalcover}
    Let $\Gamma$ be a Coxeter diagram with associated Coxeter group $W$, and let $J$ be a subset of the vertex set $\Gamma_0$. Let $\BH$ be the Brink--Howlett groupoid associated to $J$ and let $\wt{\BH}$ be the universal groupoid introduced in Definition \ref{defn:universalgroupoid}. We define a functor $\cF:\widetilde{\BH}\to \BH$ by setting
    \begin{itemize}
        \item For all $(x,I)\in \Cham(J)$,~ $\cF(x,I)\coloneq I$.
        \item For $x^{-1}y\in \Hom_{\wt{\BH}}((x,I),(y,K))$, $\cF(x^{-1}y)\coloneq x^{-1}y\in \Hom_{\BH}(I,K)$. 
    \end{itemize}
    The functor $\cF$ is a universal groupoid covering and $$\Deck(\cF) \cong N(W,J).$$
    
    
\end{Theorem}
\begin{proof}The fact that $\cF$ is a well-defined functor comes from Proposition \ref{prop:hominWCandBH}\eqref{item:WCtoBH}.
 
    Existence of path lifting follows from Proposition \ref{prop:hominWCandBH}\eqref{item:uniquelift}. Note that uniqueness comes for free since morphism spaces in $\wt{\BH}$ are singletons, from which it follows too that 
    the covering is universal.

    Next we show that the deck transformation group is $N(W,J)$. By Lemma \ref{lem:normaliseraction} we can endow $\wt{\BH}$ with an action of $N(W,J)$. Indeed, the lemma describes the action on objects, and since morphism spaces are singletons this extends uniquely to an action on the groupoid. Then $\cF$ is $N(W,J)$-equivariant, where $N(W,J)$ acts trivially on $\BH$, i.e. $N(W,J)$ acts as deck transformations of $\cF$. Since $N(W,J)$ acts faithfully on $\wt{\BH}$, this shows that $N(W,J) \leq \Deck(\cF)$.

    Conversely, suppose $\varphi \in \Deck(\cF)$. Then we obtain maps on the fibres $\varphi: \cF^{-1}(I) \to \cF^{-1}(I)$ for all associates $I$ of $J$. 
    Set $(g, J) \coloneqq\varphi(e,J) \in \cF^{-1}(J)$, so that $g$ has minimal length in $gW_J$.
    We claim that
    \[ \forall (x,I)\in \Cham(J),~ \varphi(x,I)=g.(x,I)=(gx,I).\]
    Let $(x,I)$ be a chamber, and write $(x',I)\coloneq \varphi(x,I)$. Consider the morphism $x \in \Hom_{\wt{\BH}}((e,J),(x,I))$. Under the composite functor $\cF\circ \varphi$, $x$ is mapped to $g^{-1}x' \in \Hom_{\BH}(J,I)$. Since $\varphi$ is a deck transformation this morphism must be equal to $x \in \Hom_{\BH}(J,I)$. Hence $x'=gx$, and we deduce that $\varphi(x,I) = (gx,I) = g\cdot(x,I)$ as claimed. 
    It follows that $\varphi$ agrees with the action of $gW_J \in N(W,J)$ on $\wt{\BH}$, and so $N(W,J) = \Deck(\cF)$.
\end{proof}

As a corollary we deduce a classical result of Lusztig \cite[Lemma 5.2]{Lusztig76} (which was also independently proved by Howlett \cite{Howlett80}). Set $N_J$ to be the vertex group $\BH_J$, i.e. 
$$
N_J=\{ x\in W \mid (\forall j \in J)(\exists j' \in J) \; x\cdot \alpha_j =\alpha_{j'}  \}.
$$
Note that $N_J$ is by definition a subgroup of $\norm_W(W_J)$.

\begin{Corollary}\label{cor:equivalence_brink_howlett}
The vertex group $N_J$ is isomorphic to $N(W,J)$ and moreover the normaliser of $W_J$ in $W$ is a semidirect product: $\norm_W(W_J) \cong W_J\rtimes N_J$.
\end{Corollary}

\begin{proof}
    Applying \eqref{eq:sesfornormalcover} to the universal cover $\cF:\wt{\BH}\to \BH$, we obtain that $N_J \cong \Deck(\cF)$. By Theorem \ref{thm:WCuniversalcover} this implies that $N_J \cong N(W,J) = \norm_W(W_J)/W_J$, which gives us the short exact sequence
    $$
    1 \to W_J \to \norm_W(W_J) \to N_J \to 1.
    $$
    We claim that the sequence above splits.
    For this it suffices to show that the inclusion $N_J \subseteq \norm_W(W_J)$ followed by the quotient onto $N(W,J)$ agrees with the isomorphism $N_J \cong \Deck(\cF) \cong N(W,J)$.
    To this end, note that the isomorphism $N_J \cong \Deck(\cF)$ sends $g \in N_J \mapsto \varphi \in \Deck(\cF)$ with the defining property that $\varphi(e,J) = (g,J)$.
    This in turn dictates that $\varphi \mapsto gW_J \in N(W,J)$ (see proof of Theorem \ref{thm:WCuniversalcover}), as required.
\end{proof}

\begin{Definition}[Section 2, \cite{BH}]\label{def:stdexp}
Let $I,K \in \BH$.
A  \textbf{standard expression} of $w \in  \Hom_\BH(I,K)$ is a decomposition of $w$ into a product of composable morphisms in $\BH$ of the form:
    $$
    w=v_{a_1,I_1}\cdots v_{a_m,I_m},
    $$
where $I=I_1$ and $\ell(w) = \sum_{i=1}^m \ell(v_{a_i,I_i})$. 
We also set $\ell_\BH(w)\coloneq m$ and call it the \textbf{standard expression length} of $w$. We will show below that any morphism $w \in \Hom_\BH(I,K)$ has a standard expression (Lemma \ref{lem:stdexp-geodesic}), and that its standard expression length $\ell_\BH(w)$ is independent of the choice of standard expression (Corollary \ref{cor:BHthmA}).
\end{Definition}

The following lemma relates standard expressions with geodesic paths in $(\Cone(J)^\circ,\cH^J)$.
\begin{Lemma}\label{lem:stdexp-geodesic}
    Let $\omega_{a_j,I_j}$ be simple wall crossings in $(\Cone(J)^\circ,\cH^J)$ for $1\leq j \leq m$.
    Then
    \begin{equation} \label{eqn:path}
    p \coloneqq  \omega_{a_1,I_1}*\omega_{a_2,I_2}*\cdots*\omega_{a_m,I_m}
    \end{equation}
    is a well-defined geodesic path 
    if and only if 
    \begin{equation} \label{eqn:pathstdexp}
    v_p=v_{a_1,I_1}v_{a_2,I_2}\dots v_{a_m,I_m}
    \end{equation}
    is a standard expression for $v_p$ in $\BH$ (with domain $I_1$).
    In particular, any $w \in \Hom_\BH(I,K)$ has a standard expression.
\end{Lemma}
\begin{proof}
    Suppose the path $p$ in \eqref{eqn:path} is a well-defined geodesic path from $(x,I)$ to $(y,K)$.
    Then $p$ corresponds to a sequence of simple wall-crossings relating $(x,I)$ to $(y,K)$, and hence $y=xv_p$.
    By Proposition \ref{prop:hominWCandBH}\eqref{item:WCtoBH} it follows that $v_p = x^{-1}y$ is in $\Hom_\BH(I,K)$.
    Since $p$ is geodesic, Proposition \ref{prop:geodesicsandwordlength} guarantees that \eqref{eqn:pathstdexp} is a standard expression for $v_p$.

    Conversely, suppose we have $v_p\in \Hom_\BH(I,K)$ with standard expression given by \eqref{eqn:pathstdexp}. 
    Since $\widetilde{\BH}$ is the universal cover of $\BH$ by Theorem \ref{thm:WCuniversalcover}, the sequence of morphisms in \eqref{eqn:pathstdexp} can be uniquely lifted to a sequence of composable morphisms in $\widetilde{\BH}$ with the first one having source $(x,I)$.
    By uniqueness, each $v_{a_i,I_i}$ is lifted to a morphism between adjacent chambers in $\widetilde{\BH}$ related by the wall-crossing $\omega_{a_i,I_i}$.
    This shows that the path $p$ as given in \eqref{eqn:path} is well-defined.
    The fact that $p$ is geodesic again follows from Proposition \ref{prop:geodesicsandwordlength}.

    Finally, given $w \in \Hom_\BH(I,K)$, a standard expression for $w$ can be obtained by lifting $w$ to a morphism $w \in \Hom_{\wt{\BH}}((x,I),(xw,K))$ using Theorem \ref{thm:WCuniversalcover}, where any geodesic path $p$ in $(\Cone(J)^\circ, \cH^J)$ from $(x,I)$ to $(xw,K)$ will provide a standard expression of $w = v_p$.
\end{proof}

We now deduce a part of \cite[Theorem A]{BH} from Brink--Howlett (cf.~Section \ref{sec:atomicM} below for the statement of the full theorem):
\begin{Corollary} \label{cor:BHthmA}
    \begin{enumerate}
        \item \label{item:BHgens}
        The groupoid $\BH$ is generated by morphisms of the form $$v_{a,I}:I \to I+a-\iota_{[I,a]}(a),$$ for $I \in \BH$ and $a \in I^c$ such that $W_{[I,a]}$ is a finite parabolic subgroup. 
        \item \label{item:stdexpsamelength} Let $I,K \in \BH$. Any two standard expressions of a morphism $w \in \Hom_\BH(I,K)$ are of the same length, i.e.\ if $$v_{a_1,I_1}v_{a_2,I_2}\cdots v_{a_m,I_m} = v_{b_1,L_1}v_{b_2,L_2} \dots  v_{b_n,L_n}$$ are two standard expressions with $I=I_1=L_1$, then $m=n$.
    \end{enumerate}
\end{Corollary}
\begin{proof}
    The first statement follows immediately from the groupoid covering $\cF:\wt{\BH} \to \BH$ and the fact that elements $v_{a,I}$ generate $\wt{\BH}$. Note that the finiteness condition arises from Proposition \ref{prop:simplewallcrossing}(2).

    For the second statement, let $w \in \Hom_\BH(I,K)$ and pick some $(x,I) \in \Cham(J)$. Lemma \ref{lem:stdexp-geodesic} shows that the two standard expressions of $w$ lift to two geodesic paths:
    \[
    \omega_{a_1,I_1}*\omega_{a_2,I_2} *\cdots * \omega_{a_m,I_m}; \qquad \omega_{b_1,L_1}*\omega_{b_2,L_2} * \cdots *\omega_{b_n,L_n},
    \]
    which both start $(x,I)$ (note that $I=I_1=L_1)$ and end at the same chamber $(xw,K)$.
    Since geodesic paths between two chambers have the same length, it follows that $m=n$.
\end{proof}

\begin{Remark}\label{rmk:BHlen}
Let $x=yz$ be an equality of morphisms in $\BH$. Suppose that  $\ell(x)=\ell(y)+\ell(z)$ and $\ul{y},\ul{z}$ are respective standard expressions for $y,z$. Then it is immediate that the concatenation of $\ul{y}$ and $\ul{z}$ is a standard expression of $x$, which we denote by $\ul{y}\ul{z}$.
\end{Remark}

\subsection{Brink and Howlett's Theorem A and the atomic Matsumoto theorem}\label{sec:atomicM}
The result \cite[Theorem A]{BH} provides a presentation by generators and relations for the groupoid $\BH$. The generators are precisely those that appear in Corollary \ref{cor:BHthmA}, and the relations are similar to that of a Coxeter group presentation (see Remark \ref{rem:analogy_brink_howlet_coxeter_presentation}). In this section we will further this analogy by showing an analogue of Mastumoto's theorem on reduced expressions for this presentation.

Before we are able to describe the presentation, we recall the following result from \cite{BH}. 

\begin{Lemma} \label{lem:BHuniquelcm}
    Let $I$ be an associate of $J$.
    Suppose we have distinct elements $a,b \in I^c$ such that $|\Phi_{I+a+b} - \Phi_I| < \infty$.
    Then there exists an associate $K$ of $J$ and a morphism $v_{a,b,I}$ in $\Hom_{\BH}(I,K)$ such that:
    \begin{enumerate}
        \item \label{item:stdexp} $v_{a,I}$ and $v_{b,I}$ are left divisors of $v_{a,b,I}$ in $W$, and $v_{a,b,I}$  has exactly two standard expressions (of necessarily the same length):
                \begin{equation} \label{eqn:lcmstdexp}
                v_{a,b,I} =
                \begin{cases}
                v_{a_1,I_1}v_{a_2,I_2}\cdots v_{a_m,I_m},\\
                v_{b_1,L_1}v_{b_2,L_2}\cdots v_{b_m,L_m},
                \end{cases}
                \end{equation}
              with $a_1=a, b_1=b$ and $I_1=I=L_1$;
        \item \label{item:lcm} If $w \in \Hom_{\BH}(I,K')$ admits both $v_{a,I}$ and $v_{b,I}$ as left divisors, then $v_{a,b,I}$ is a left divisor of $w$.
    \end{enumerate}
    In other words, $v_{a,b,I}$ is minimal (with respect to the standard expression length and with respect to left divisibility) among the morphisms in $\BH$ of which $v_{a,I}$ and $v_{b,I}$ are both left divisors.
\end{Lemma}
\begin{proof}
Under the assumption $|\Phi_{I+a+b} - \Phi_I| < \infty$, the existence of $v_{a,b,I}$ and its two standard expressions in \eqref{eqn:lcmstdexp} follow from the discussion on ``\textit{Type R2}'' relations in \cite[pg 326]{BH} (precisely, $v_{a,b,I}^{-1} = v[\{a,b\},I]$ in the notation of \cite{BH}; see Remark \ref{rmk:oppositev}).
This combined with \cite[Proposition 2.4]{BH} gives statement \eqref{item:stdexp}.
Statement \eqref{item:lcm} is a translation of \cite[Lemma 4.1]{BH}.
\end{proof}
\begin{Remark}
    The inverse of $v_{a,b,I}$  is characterised as the unique element $w \in W_{I+a+b}$ satisfying $w(\Phi^+_{I+a+b} - \Phi^+_I) \subseteq \Phi^-_{I+a+b}$ and $w(\Phi^+_I) \subseteq \Phi^+_{I+a+b}$.
    See discussion after Proposition 2.1 in \cite{BH}.
\end{Remark}



\begin{Definition} \label{def:atomicbraidrel}
In the setting of Lemma \ref{lem:BHuniquelcm} the relation given by equating the two expressions on the right hand side of \eqref{eqn:lcmstdexp} is called an \textbf{atomic braid relation}.  
\end{Definition}

We can now state \cite[Theorem A]{BH}.

\begin{Theorem}[Theorem A, \cite{BH}]\label{thm:BHthmA}
    Let $\Gamma$ be a Coxeter diagram with associated Coxeter group $W$, and let $J$ be a subset of the vertex set $\Gamma_0$. The Brink--Howlett groupoid $\BH$ associated to $J$ is generated by the morphisms $v_{a,I}$ from Corollary \ref{cor:BHthmA}(\ref{item:BHgens}), with defining relations given by all the atomic braid relations and the quadratic relations $v_{a,I}v_{a',K} = e$ for $a' = \iota_{[I,a]}(a)$ and $ K = I + a - a'$ (cf.\ Proposition \ref{prop:simplewallcrossing}).
\end{Theorem}

\begin{Remark}\label{rem:analogy_brink_howlet_coxeter_presentation}
    In the presentation above, the relations are similar to quadratic relations and braid relations for the Coxeter presentation of Coxeter groups. 
    This analogy is not a coincidence, as the presentation of Brink--Howlett groupoids generalises the presentation of Coxeter groups. Indeed if we take $J=\varnothing$, then for every $a\in \Gamma_0$, we have $v_{a,\emptyset}=s_a$ by Remark \ref{rem:v_{a,varnothing}}. The morphism $v_{a,b,I}=v_{a,b,\emptyset}$ exists if and only if $ab$ has finite order (say $m$) in $W$, and $v_{a,b,\emptyset}$ is then given by
    \[v_{a,b,\emptyset}=s_as_bs_a\cdots=s_bs_as_b\cdots.\]
    where both words have length $m$. 
\end{Remark}

For Coxeter groups, the classical Matsumoto theorem states that any two reduced expressions for an element $w \in W$ are related by a sequence of braid relations. We now prove a generalisation of this for standard expressions in the sense of Definition \ref{def:stdexp}. 

\begin{Theorem}[``Atomic Matsumoto theorem''] \label{thm:atomicMat}
Let $\Gamma$ be a Coxeter diagram, and let $J$ be a subset of the vertex set $\Gamma_0$. Let  $\BH$ be the Brink--Howlett groupoid associated to $J$, and let $I,K$ be associates of $J$. Any two standard expressions of $w \in \Hom_{\BH}(I,K)$ are related by a sequence of atomic braid relations.
\end{Theorem}

\begin{proof}
    Using Lemma \ref{lem:BHuniquelcm}, the proof is essentially the same as the proof for the classical case; cf.\ \cite[Theorem 3.3.1]{BB_combinatoricsCox}.  We induct on $\ell_\BH(w)$. If $\ell_\BH(w)=1$ then  $w$ has only one standard expression and the result is trivial. Suppose then that $\ell_\BH(w)>1$ and the result holds for any $x$ such that  $\ell_\BH(x) < \ell_\BH(w)$. 
    
    Consider two standard expressions for $w$, the first of which starts with $v_{a,I}$ and the second starts with $v_{b,I}$ (note that any standard expression for $w$ has to start with $v_{?,I}$). We denote the two standard expressions for $w$ by $v_{a,I}\ul{w}_1$ and $v_{b,I}\ul{w}_2$, so here $\ul{w}_i$ denotes a morphism with a chosen standard expression for $w_i$.
    Then $v_{a,I}$ and $v_{b,I}$ are both left divisors of $w$, and hence by Lemma \ref{lem:BHuniquelcm}, $v_{a,b,I}$ is also a left divisor of $w$. Set $u=v_{a,b,I}^{-1}w$ (so that $u \in \Hom_\BH(K',K)$ if $v_{a,b,I} \in \Hom_\BH(I,K')$).
    
    Let $v_{a,I}\ul{x}_1$ and $v_{b,I}\ul{x}_2$ be the two standard expressions for $v_{a,b,I}$ from Lemma \ref{lem:BHuniquelcm}. Let $\ul{u}$ be any standard expression for $u$.
     By Remark \ref{rmk:BHlen} $\ul{w}_1$ and $\ul{x}_1\ul{u}$ are both standard expressions for $v_{a,I}^{-1}w$. 
    Therefore by induction $\ul{w}_1$ is related to $\ul{x}_1\ul{u}$ by a sequence of atomic braid relations. Therefore $v_{a,I}\ul{w}_1$ and $v_{a,I}\ul{x}_1\ul{u}$ are also related by a atomic braid relations. Similarly $v_{b,I}\ul{w}_2$ and $v_{b,I}\ul{x}_2\ul{u}$ are also related by atomic braid relations. Since $v_{a,I}\ul{x}_1\ul{u}$ and $v_{b,I}\ul{x}_2\ul{u}$ are also clearly related by atomic braid relations, we conclude the same for $v_{a,I}\ul{w}_1$ and $v_{b,I}\ul{w}_2$.    
\end{proof}

\begin{Remark}[Relation of Theorem \ref{thm:atomicMat} to Ko's \protect{\cite[Theorem 1.8]{Ko25}}]\label{rem:relation_to_ko}
    The atomic Matsumoto theorem was stated in \cite{Ko25} in terms of core and atomic (double) cosets.
    We only provide a brief translation here, referring the keen reader to \cite{Ko25} for details.
    
    A \textit{core (double) coset} \cite[Definition 1.1]{Ko25} is a double coset $W_IxW_K$ satisfying $W_Ix = W_IxW_K = xW_K$; cf.\ \eqref{eqn:chamberlabels}.
    An \textit{atomic (double) coset} \cite[Definition 1.6]{Ko25} is a core coset of the form $W_{L-a}w_LW_{L-a'}$ for $a,a' \in L \subseteq \Gamma_0$.
    To match our notation, set $I \coloneqq L-a$, so that $L = I+a$ and $L-a' = I+a-a'$. 
    The core coset condition guarantees that $a' = \iota_{[I,a]}(a)= \iota_{I+a}(a)$; the second equality holds by Lemma \ref{lem:combinatoriallemmas} under the condition that $W_{I+a} = W_L$ is finite, which is always assumed in \cite{Ko25}.
    
    Now, it is clear that $w_L$ is not the minimal length representative of the atomic coset $W_{L-a}w_LW_{L-a'}$. This is instead given by $v_{a,I}= w_Iw_{I+a}=w_{L-a}w_L$.
    More generally, by taking minimal double coset representatives, the set of core cosets of the form $W_IxW_K$ are in bijection with $\Hom_\BH(I,K)$, with the atomic cosets $W_Iw_{I+a}W_{I+a-a'}$ corresponding to the generators $v_{a,I}$.
    Under this correspondence: 
    \begin{itemize}
        \item An \textit{atomic reduced expression} \cite[Definition 3.9]{Ko25} of a core coset $W_IxW_K$ corresponds to a standard expression of $x \in \Hom_\BH(I,K)$ (with $x$ minimal in $W_IxW_K$).
        \item An \textit{atomic braid relation} \cite[Definition 4.10]{Ko25} for core cosets corresponds to our Definition \ref{def:atomicbraidrel} for morphisms in $\BH$.
        \item The \textit{atomic Matsumoto theorem} \cite[Theorem 1.8]{Ko25} for atomic reduced expressions corresponds to our Theorem \ref{thm:atomicMat} for standard expressions. Note, however, that our theorem is more general, as we do not require the finitary assumption assumed in \cite{Ko25}, i.e.\ $W_L=W_{I+a}$ need not be finite.
    \end{itemize}
\end{Remark}

We also translate the results above into the language of wall-crossings and hyperplane arrangements, as they answer a question posed by Iyama--Wemyss in \cite[Remark 1.62]{IW}.
We begin with the following translation of Lemma \ref{lem:BHuniquelcm}.
\begin{Lemma}\label{lem:arrangementBHlcm}
    Let $a,b \in I^c$
    be distinct elements such that $|\Phi_{I+a+b} - \Phi_I| < \infty$, and let $(x,I) \in \Cham(J)$. Then there exists $K \subseteq \Gamma_0$ such that:
    \begin{enumerate}
        \item \label{item:meet} In the hyperplane arrangement $(\Cone(J)^\circ, \cH^J)$, we have two sequences of wall-crossings of common length that starts at $(x,I)$ and ends at $(xv_{a,b,I},K)$:
        \begin{equation} \label{eqn:meet}
        \begin{tikzcd}[row sep=small, column sep = small]
	   & {(x_2,I_2)} & {(x_3,I_3)} & \cdots & {(x_{m},I_{m})} \\
	   {(x,I)} &&&&& {(xv_{a,b,I},K)} \\
	   & {(y_2,L_2)} & {(y_3,L_3)} & \cdots & {(y_{m}, L_{m})}
	   \arrow[from=1-2, to=1-3]
	   \arrow[from=1-3, to=1-4]
	   \arrow[from=1-4, to=1-5]
	   \arrow[from=1-5, to=2-6]
	   \arrow["{\omega_{a_1,I}}", from=2-1, to=1-2]
	   \arrow["{\omega_{b_1,I}}"', from=2-1, to=3-2]
	   \arrow[from=3-2, to=3-3]
	   \arrow[from=3-3, to=3-4]
	   \arrow[from=3-4, to=3-5]
	   \arrow[from=3-5, to=2-6]
        \end{tikzcd}.
        \end{equation}
        Here $x_{i+1} = x_{i}v_{a_i,I_i}$ and $y_{i+1} = y_{i}v_{b_i,L_i}$, with $x_1=y_1=x$ and $I_1=L_1=I$, and $$v_{a_1,I_1}\cdots v_{a_m,I_m},\quad v_{b_1,L_1}\cdots v_{b_m,L_m}$$ are the two standard expressions of $v_{a,b,I}$ in \eqref{eqn:lcmstdexp}.
        Moreover, these two sequences, viewed as paths in $(\Cone(J)^\circ, \cH^J)$,  are the only two geodesic paths from $(x,I)$ to $(xv_{a,b,I},K)$.
        \item \label{item:mingeodesic} 
        Let $C$ be a chamber in $(\Cone(J)^\circ,\cH^J)$ such that there exists geodesic paths $p_a = \omega_{a,I}*p'_a:(x,I) \to C$ and $p_b = \omega_{b,I}*p'_b: (x,I) \to C$.
        Then there exists a geodesic path $p:(x,I) \to C$ that passes through the chamber $(xv_{a,b,I},K)$.
        \end{enumerate}
\end{Lemma}
\begin{proof}
    This is a direct translation of Lemma \ref{lem:BHuniquelcm} using the correspondence between standard expressions in $\BH$ and geodesic paths in $(\Cone(J)^\circ,\cH^J)$ from Lemma \ref{lem:stdexp-geodesic}.
\end{proof}

The following is the analogue of atomic braid relation in this setting; cf.\ Definition \ref{def:atomicbraidrel}.
\begin{Definition}[\protect{\cite[Definition 1.58]{IW}}] \label{def:codim2rel}
In the setting of Lemma \ref{lem:arrangementBHlcm}, the relation given by equating the two geodesic paths in \eqref{eqn:meet} is called a \textbf{codimension two relation}.  
\end{Definition}

The atomic Matsumoto theorem (Theorem \ref{thm:atomicMat}) then translates to the following (again via Lemma \ref{lem:stdexp-geodesic}):
\begin{Corollary} \label{cor:arrangement_atomicMat}
    Any two geodesic paths $p,p':(x,I) \to (y,K)$ in $(\Cone(J)^\circ,\cH^J)$ are related by a finite sequence of codimension two relations.
\end{Corollary}
\begin{Remark}
 Corollary \ref{cor:arrangement_atomicMat} allows one to generalise \cite[Theorem 1.60]{IW} to the ``uniformly weakly Dynkin setting'', answering a question posed by Iyama and Wemyss. See \cite[Remark 1.62]{IW} for the precise details.
\end{Remark}

\section{The reduced ribbon groupoid} \label{sec:reducedribgroupoid}
In this section we recall the construction of the reduced ribbon groupoid associated to a standard parabolic subgroup of a spherical-type Artin--Tits group, which will play a crucial role in our proof of Theorem \ref{thm:main1v2}. This groupoid was introduced by Godelle in the more general setting of standard parabolic subgroup of a Garside group. We stress that we need to impose the spherical-type assumptions from now on.

In this section, we fix a spherical Coxeter diagram $\Gamma$ with associated Coxeter group $W$. We also fix a subset $J$ of the vertex set $\Gamma_0$. We otherwise keep the notation of the previous sections. 

We first recall the definition of an auxiliary (larger) groupoid:
\begin{Definition}[\cite{Godelle_RibbonII}]
   Let $\Gamma$ be a spherical Coxeter diagram and let $J$ be a subset of the vertex set $\Gamma_0$. The \textbf{ribbon groupoid}  $\R' \coloneqq \R'(J)$ associated to $J$ is defined as follows.
    \begin{itemize}
        \item The objects are the subsets of $\Gamma_0$ which are associates of $J$ (see Definition \ref{def:associate}).
        \item $\Hom_{\R'}(I,K) \coloneqq \{ \beta \in A \mid A_I^+\beta=\beta A_K^+ \}$. Here $A_I^+$ is the submonoid of $A_I$ generated by the $\sigma_{i}$ for $i\in I$. Composition is given by multiplication in $A$ (on the right).
    \end{itemize}
\end{Definition}

Let $I$ be a subset of $\Gamma_0$ and let $a\in I^c$. By sphericality, the parabolic subgroup $W_{[I,a]}$ is finite and we define
\begin{equation}\label{eqn:redribgen}
    \mu_{a,I} \coloneqq \Delta^{-1}_{[I,a]-a}\Delta_{[I,a]} \in A^+,
\end{equation}
where $\Delta_K \in A^+$ denotes the positive lift of the longest element $w_K \in W_K$ (see Notation \ref{notation:longest_element}).
Note that we are taking the reduced expression in \eqref{eqn:redribgen}, so that it is an element in $A^+$.
Moreover, \eqref{eqn:redribgen} is the positive lift of \eqref{eqn:BHgen}, i.e.\ it is the image under the set-theoretic section $W \to A^+$ of the projection map $\pi:A\to W$. 

Note that by sphericality, \eqref{eqn:redribgen} reduces to the following
\begin{equation}\label{eqn:finiteredribgen}
    \mu_{a,I} = \Delta^{-1}_I\Delta_{I+a} \in A^+;
\end{equation}
c.f.\ \eqref{eqn:reducedwallcrossingformula} in Lemma \ref{lem:combinatoriallemmas}. One can readily check that 
\begin{equation}\label{eq:Rnugens}
\mu_{a,I} \in \Hom_{\sR'}(I,I+a-\iota_{[I,a]}(a)).
\end{equation}

\begin{Definition}[\protect{\cite{Godelle_RibbonII}}] \label{def:reduced_ribbon_groupoid}
    Let $\Gamma$ be a spherical Coxeter diagram, and let $J$ be a subset of the vertex set $\Gamma_0$. The \textbf{reduced ribbon groupoid} $\R \coloneqq \R(J)$ associated to $J$ is the subgroupoid of the ribbon groupoid $\R'$ whose objects coincide with the objects of $\R'$, and whose morphisms are generated by the morphisms $\mu_{a,I}$ for $I$ an associate of $J$ and $a \in I^c$.
\end{Definition}

The relevance of $\R$ for us is that its vertex group $\R_J$ is isomorphic to $N(A,J)$:
\begin{Theorem}[\protect{\cite[Proposition 4.4]{GodelleI}}] \label{thm:Artinnormaliser} Let $\Gamma$ be a spherical Coxeter diagram with associated Artin--Tits group $A$, and let $J$ be a subset of the vertex set $\Gamma_0$. Let $\sR$ be the reduced ribbon groupoid associated to $J$. 
    There is an isomorphism of groups
    \[
    \norm_A(A_J) \cong A_J \rtimes \R_J.
    \]
    In particular, the vertex group $\R_J$ is isomorphic to the normaliser quotient $N(A,J)$.
\end{Theorem}

\begin{Remark}\label{rem:Godpresent}
   In the spherical case, the monoid $A^+$ has a Garside structure. This implies that $A^+$ admits all greatest common divisors and least common multiples relative to both left and right divisibility. In particular, for any associate $I$ of $J$, and $a,b\in I^c$, we can consider the least common multiple (for left divisibility) $\mu_{a,I}\vee \mu_{b,I}$ in $A^+$. 

    By \cite[Proposition 4.19]{Godelle_RibbonII}, $\mu_{a,I} \vee \mu_{b,I}$ is a morphism in $\cR$ and it has precisely two expressions as a product of generating morphisms:
\begin{equation} \label{eqn:lcmcompletion}
        \mu_{a,I} \vee \mu_{b,I}=
        \begin{cases}
        \mu_{a_1,I_1}\mu_{a_2,I_2}\cdots \mu_{a_m,I_m},\\\mu_{b_1,L_1}\mu_{b_2,L_2}\cdots \mu_{b_n,L_n},
        \end{cases}
\end{equation}
where $a_1=a$, $b_1=b$ and $I_1=I=L_1$. Although we won't use it for our main results, we note that Godelle \cite[Corollary 4.18]{Godelle_RibbonII} proved that the relations obtained by equating the right hand side of \eqref{eqn:lcmcompletion} is a complete set of relations for $\cR$.  
\end{Remark}

Notice that the two expressions of the least common multiple $\mu_{a,I}\vee \mu_{b,I}$  in \eqref{eqn:lcmcompletion} are similar to the two standard expressions of the element $v_{a,b,I}$ in \eqref{eqn:lcmstdexp}. The following lemma gives a precise relationship.

\begin{Lemma}\label{lem:stdexp}
    Let $I$ be an associate of $J$ and let $a,b\in I^c$.
    The two expressions of the least common multiple $\mu_{a,I}\vee \mu_{b,I} \in A^+$ in \eqref{eqn:lcmcompletion} agree with the positive lifts of the two standard expressions of $v_{a,b,I}$ in \eqref{eqn:lcmstdexp} of Lemma \ref{lem:BHuniquelcm}.
    In particular, we have $m=n$ and $\pi(\mu_{a,I}\vee \mu_{b,I}) =v_{a,b,I}$.
    Moreover, $v_{a,b,I}$ is the join of $v_{a,I}$ and $v_{b,I}$ for the weak left Bruhat order on $W$.
\end{Lemma}
\begin{proof}
    Denote by $\wt{W}\subset A^+$ the set of positive lifts of elements of $W$ in $A^+$. Since $\Gamma$ is a spherical-type diagram, by \cite[Section IX.1.3]{ddgkm} the map $\pi: A \to W$ restricts to an isomorphism between the following posets:
    \begin{enumerate}
        \item $\wt{W}$: endowed with left-divisibility in the monoid $A^+$, and 
        \item $W$: endowed with left-divisibility from the weak left Bruhat order. 
    \end{enumerate}
    Moreover, the inverse isomorphism is the positive lift $W\to A^+$. 
    In particular, if $w=s_1s_2\cdots s_r$ is a length-additive decomposition of $w$ (i.e.\ a reduced expression), then the product $\sigma_{s_1}\sigma_{s_2}\cdots\sigma_{s_r}$ in $A^+$ is a decomposition of the positive lift $\sigma_w$ in $\wt{W}\subset A^+$ .

    The condition in Lemma \ref{lem:BHuniquelcm} is always satisfied since $\Gamma$ is of spherical-type by assumption. Now, note that each $\mu_{a,I}\in A^+$ is the positive lift of $v_{a,I}\in W$ by construction. In particular, $\mu_{a,I}\vee \mu_{b,I}$ lies in $\wt{W}$. 
    By applying the isomorphism of posets $\pi: \wt{W} \to W$ to its two expressions as in \eqref{eqn:lcmcompletion}, we obtain two standard expressions of $\pi(\mu_{a,I}\vee\mu_{b,I})$.
    Thus $\pi(\mu_{a,I}\vee\mu_{b,I})$ is left-divided by both $v_{a,I}$ and $v_{b,I}$, and it is moreover a morphism in $\BH$ (starting from $I$). By Lemma \ref{lem:BHuniquelcm}, $v_{a,b,I}$ left divides $\pi(\mu_{a,I}\vee \mu_{b,I})$.
    
    Conversely, since the two decompositions of $v_{a,b,I}$ in \eqref{eqn:lcmstdexp} are standard expressions, they both lift to decompositions of the positive lift of $v_{a,b,I}$ in $\wt{W} \subset A^+$. In particular, $\mu_{a,I}$ and $\mu_{b,I}$ both left-divide the positive lift of $v_{a,b,I}$ in $A^+$, and so does their lcm $\mu_{a,I}\vee \mu_{b,I}$.
    We then obtain that $\mu_{a,I}\vee \mu_{b,I}$ and the positive lift of $v_{a,b,I}$ left-divide one another in $A^+$, and thus they are equal. 

    The rest of the statement is a direct consequence of the isomorphism of posets $\pi:\wt{W}\to W$.
\end{proof}

\section{Normaliser quotients and fundamental groups}\label{sec:mainresult}
This section is devoted to the proof of Theorem \ref{thm:main1v2}, which will be an immediate consequence of Theorem \ref{thm:main} below. As this will require results of the last section, we will still assume sphericality.

Fix a spherical Coxeter diagram $\Gamma$ with associated Coxeter group $W$. We also fix a subset $J$ of the vertex set $\Gamma_0$. We otherwise keep the notation of the previous sections. 

To begin, we construct a functor $\cG: \D \to \R$ as follows:
\begin{itemize}
    \item For all object $(x,I)\in \Cham(J)$, we set $\cG(x,I)=I$.
    \item For a generating morphism $\omega_{a,I}:(x,I) \to (y,K)$, we define $\cG(\omega_{a,I})=\mu_{a,I} \in \Hom_\R(I,K)$.
\end{itemize}
We show below that this mapping does define a functor.

Next, recall we have an action of $N(W,J)$ on $\D$ which is free on the objects of $\D$ (see Remark \ref{rmk:NormactiononD}). Let $\cQ:\D \to \D/N(W,J)$ denote the quotient functor.
The essential feature of $\cG$ which will imply Theorem \ref{thm:main1v2}  is that it is a groupoid covering isomorphic to $\cQ$. We will build up to this in a series of results.

\begin{Lemma}\label{lem:Gwelldef}
    The functor $\cG:\D \to \R$ is well-defined.    
\end{Lemma}
\begin{proof}
    Since $\D$ is the groupoid completion of $\D^+$ and $\cG$ maps into a groupoid, it is sufficient to prove that $\cG$ is well-defined on $\D^+$.
    In particular, we only have to show that any two geodesic paths in $\D^+$ with the same source and target are sent to the same morphism in $\R$.
    To do so, we will prove that any geodesic path in $\D^+$ starting at $(x_1,I_n)$ and ending at $(x_n,I_n)$ is sent to the positive lift of $x_1^{-1}x_n \in W$.
    This is sufficient since this description depends only on the source and target of the path.
    
    So suppose $p$ is a geodesic path in $\D^+$ traveling through the following sequence of chambers
    \[
    p: (x_1,I_1) \to (x_2,I_2) \to \cdots \to (x_n,I_n)
    \]
    with $(x_{i+1}, I_{i+1}) = \omega_{a_i}(x_i,I_i)$ for all $1 \leq i \leq n-1$.
    Under $\cG$, $p$ is sent to the element
    \[
    \cG(p) = \mu_{a_1,I_1}\mu_{a_2,I_2}\ldots \mu_{a_{n-1},I_{n-1}} \in \Hom_{\R}(I_1,I_n).
    \]
    Recall that each $\mu_{a_i,I_i}$ is, by definition, the positive lift of $v_{a_i,I_i}$. 
    Since $p$ is geodesic, by Proposition \ref{prop:geodesicsandwordlength} we have that $v_{a_1,I_1}v_{a_2,I_2}\cdots v_{a_{n-1},I_{n-1}}$ is a standard expression for $x_1^{-1}x_n$.
    Consequently, $\cG(p)$ is indeed the positive lift of $x_1^{-1}x_n$, as required.
\end{proof}

\begin{Lemma}\label{lem:Gcover}
    The functor $\cG:\D \to \R$ is a groupoid covering.
\end{Lemma}
\begin{proof}
    The existence of path lifting is immediate.
    Indeed, any generating path $$\mu_{a,I} \in \Hom_{\R}(I,I+a-\iota_{[I,a]}(a))$$ can be lifted to a one-step path $\omega_{a,I}:(x,I) \to \omega_a(x,I)$ in $\D^+$ for any choice of $(x,I) \in \Cham(J)$.

    To prove uniqueness, it suffices to show that the two sides of any generating relation in \eqref{eqn:lcmcompletion} lift 
    to the same path in $\D^+$ once we fix a starting point.

    Suppose we have $\mu_{a,I}$ and $\mu_{b,I}$ as in \eqref{eqn:lcmcompletion}, and let $(x,I) \in \Cham(J)$. Then we can lift the two products in \eqref{eqn:lcmcompletion} to paths in $\D^+$ starting at $(x,I)$. By Proposition \ref{prop:geodesicsandwordlength}, since these products both descend to standard expressions of an element $w \in W$ (Lemma \ref{lem:stdexp}), both of these paths are geodesic paths. Moreover, these paths have the same endpoint, given by $(xw,K)$ for some $K \in \R$. By the defining relations in $\D^+$, the paths are equal. This proves uniqueness of path-lifting. 
\end{proof}

\begin{Proposition}\label{prop:G'equiv}
    The functor $\cG$ factors through $\cQ$, resulting in a functor $$\cG':\D/N(W,J) \to \R.$$
    Moreover, $\cG'$ is an isomorphism of groupoids.
\end{Proposition}

\begin{proof}
Set $N\coloneqq N(W,J)$ for simplicity. 
    Endow $\R$ with the trivial $N$-action. Then showing that $\cG$ factors is equivalent to show that $\cG$ is $N$-equivariant. 
On objects, this is clear since $N$ acts only on the first factor of chamber labels $(x,I) \in \Cham(J)$, and $\cG:(x,I) \mapsto I$.
On morphisms, equivariance is also immediate since both $\omega_{a,I}:(x,I) \to (y,K)$ and its image under $g\in N$, i.e. $\omega_{a,I}:g\cdot(x,I) \to g\cdot(y,K)$, are mapped by $\cG$ to $\mu_{a,I}$. Since these morphisms generate $\sD$ this implies that $\cG$ is $N$-equivariant.

Next we show that the induced functor $\cG':\D/N \to \R$ is an isomorphism. Since both $\sD/N$ and $\R$ are small groupoids, it is sufficient to show that $\cG'$ is fully faithful and bijective on objects. 

The fact that $\cG'$ is bijective on objects is straightforward: by Lemma \ref{lem:normaliseraction}, the map $(x,I)\mapsto I$ induces a bijection between $\Cham(J)/N$ and the associates of $J$, and this is exactly the map of $\cG'$ on objects by definition.

Fullness follows from the equality $\cG'\circ \cQ = \cG$ and path-lifting. Indeed, given a morphism $\mu:I \to K$ in $\R$, there exists a lift $\tilde{\mu}:(x,I) \to (y,K)$, since $\cG$ is a cover. Set $\mu'\coloneq \cQ(\tilde{\mu})$. Then $\cG(\mu')=\mu$. 

Finally, we show that $\cG'$ is faithful. Let $[(x,I)], [(y,K)] \in \D/N$, and consider $$[p],[p'] \in \Hom_{\D/N}([(x,I)], [(y,K)]),$$
such that $\cG'([p])=\cG'([p'])$.
Since $N$ acts transitively on first labels of elements in $\Cham(J)$ (cf.\ Lemma \ref{lem:normaliseraction}(2)), we can choose representatives $p,p'$ such that 
\begin{align*}
    p:(x,I) &\to g\cdot (y,K),\\
    p':(x,I) &\to g'\cdot (y,K),
\end{align*}
where $g,g' \in N$.
Then we have that $\cG(p)=\cG(p')$. But now $p$ and $p'$ are both lifts of $\cG(p)=\cG(p')$ with the same source. By uniqueness of path lifting, we have $p=p'$, and hence $[p]=[p']$.
\end{proof}

\begin{Corollary}\label{cor:GandQequiv}
    The covers $\cG$ and $\cQ$ are isomorphic. In particular, $\cG$ is a normal covering and its deck transformation group of is isomorphic to $N(W,J)$.
\end{Corollary}

\begin{proof}
By Proposition \ref{prop:G'equiv} we have a commutative diagram
\[
\begin{tikzcd}
    & \D \ar[dl, "\cQ"'] \ar[dr, "\cG"] & \\
  \D/N(W,J) \ar[rr, "\cG'"]  &  & \R
\end{tikzcd}
\]
with $\cG'$ an isomorphism. Hence the two covers are isomorphic. Since $\cQ$ is normal and its group of deck transformations is isomorphic to $N(W,J)$, the same is true for $\cG$.
\end{proof}

\begin{Remark}\label{rem:construction_diagram_global}
We can now complete the square in \eqref{eq:global_diagram}.
The top horizontal arrow is described in Remark \ref{rem:obvious_functor}. The left vertical arrow is the functor $\cG$. The right vertical arrow is the universal cover of the Brink--Howlett groupoid obtained in Theorem \ref{thm:WCuniversalcover}. The bottom horizontal arrow is the functor mapping the generator $\mu_{a,I}$ of $\sR$ to the generator $v_{a,I}$ of $\BH$ (this is a well-defined functor by Godelle's presentation of $\R$, cf.\ Remark \ref{rem:Godpresent}).
    To show that the square is commutative it is sufficient to show commutativity on a set of generating morphisms of $\sD$, namely the simple wall crossings. Let $\omega_{a,I}$ be a simple wall crossing between two chambers $(x,I)$ and $(y,K)$. By definition, we have $\cQ(\omega_{a,I})=\mu_{a,I}$, and the bottom horizontal arrow of the diagram maps $\mu_{a,I}$ to $v_{a,I}$. Then, the top horizontal arrow maps $\omega_{a,I}$ to the unique morphism in $\widetilde{\BH}$ from $(x,I)$ to $(y,K)$. The universal covering $\widetilde{\BH}\to \BH$ maps this unique morphism to $x^{-1}y\in \hom_{\BH}(I,K)$. Since $(x,I)$ and $(y,K)$ are related by a simple wall crossing $\omega_{a,I}$, we have $x^{-1}y=v_{a,I}$.
\end{Remark}


Recall from Section \ref{sec:Dgroupoid} we have an equivalence of groupoids $\D \cong \pi_1(\cX_J^\CC)$. We need that the equivalence can be made equivariant with respect to the action of $N(W,J)$ (see discussion before \eqref{eq:TitsconeSES} for the action of $N(W,J)$ on $\cX_J^\CC$). This follows from the same argument one uses for Coxeter arrangements ($J=\emptyset)$. For the sake of completeness, we include an argument in the following lemma using results from Van der Lek's thesis \cite{VdL_thesis}. 
\begin{Lemma}[An $N(W,J)$-equivariant equivalence]\label{lem:equiv_quotient_pi1}
    There is an $N(W,J)$-equivariant equivalence $p:\D\to \pi_1(\cX_J^\CC)$, which induces the following commutative square
\[\begin{tikzcd}
	\D & {\pi_1(\cX_J^\CC)} \\
	{\D/N(W,J)} & {\pi_1(\cX_J^\CC/N(W,J))}
	\arrow["p", from=1-1, to=1-2]
	\arrow["\cQ"', from=1-1, to=2-1]
	\arrow[from=1-2, to=2-2]
	\arrow["{\overline{p}}"', from=2-1, to=2-2]
\end{tikzcd}\]
where the right vertical map is induced by the covering $\cX_J^\CC\to \cX_J^\CC/N(W,J)$, and the functor $\overline{p}$ is an equivalence of groupoids.
\end{Lemma}
\begin{proof}
We only have to construct an $N(W,J)$-equivariant equivalence $p:\D\to \pi_1(\cX_J^\CC)$.
Indeed, the normal covering $\cX_J^\CC\to \cX_J^\CC/N(W,J)$ gives a groupoid covering $\pi_1(\cX_J^\CC)\to \pi_1(\cX_J^\CC/N(W,J))$ which agrees with the groupoid covering $\pi_1(\cX_J^\CC)\to \pi_1(\cX_J^\CC)/N(W,J)$ onto the quotient category; see Section \ref{sec:groupoid_coverings}.
The $N(W,J)$-equivariance then guarantees that we have a well-defined equivalence $\overline{p}: \D/N(W,J) \to \pi_1(\cX_J^\CC)/N(W,J) = \pi_1(\cX_J^\CC/N(W,J))$.

We now provide the construction of such an equivalence.
For each chamber $C$ in $(\Cone(J)^\circ,\cH^J)$, define
\[X_C:=\{x+iy\in \cX_J^\CC~|~ \forall H\in \cH^J, x\in H\Rightarrow y\in D_H(C)\},\]
where $D_H(C)$ denotes the half space delimited by $H$ in which $C$ is included. The set $X_C$ is open in $\cX_J^\CC$ and is contractible by \cite[Corollary I.2.18]{VdL_thesis}. Let $C,C'$ be two adjacent chambers separated by a hyperplane $H_0$. By \cite[Corollary I.2.19]{VdL_thesis}, the intersection $X_C\cap X_{C'}$ admits two connected components, given by
\begin{align*}
    X_{C,C'}\coloneq\{x+iy\in \cX_J^\CC~|~x\in D_{H_0}(C)\text{~and~} \forall H\in \cH^J,x\in H\Rightarrow y\in D_H(C)=D_H(C')\},
\end{align*}
and $X_{C',C}$, which is defined by swapping the roles of $C$ and $C'$ above. 
Define a functor $p:\D\to \pi_1(\cX_J^\CC)$ as follows. Firstly, choose for each chamber $C$ a point $a_C \in X_C$ in such a way that $a_{g\cdot C}=g\cdot a_C$. This defines $p$ on objects, sending $C \mapsto a_C$. 
For adjacent chambers $C$ and $C'$, the image of the one-step path $C\to C'$ in $\D$ under $p$ is defined as follow.
Choose a point $d\in X_{C,C'}$ and let $\gamma_1$ (resp.\ $\gamma_2$) be a path in $X_{C}$ (resp.\ in $X_{C'}$) from $a_C$ to $d$ (resp.\ from $d$ to $a_{C'}$).
The concatenation $\gamma_1\gamma_2$ is a path from $a_C$ to $a_{C'}$ in $\cX_J^\CC$, whose homotopy class depends only on the connected component $X_{C,C'}$ (and does not depend on the choice of $\gamma_1,\gamma_2,d$); see discussion after \cite[Definition I.1.3]{VdL_thesis}.
The image of the one-step path $C\to C'$ in $\D$ under $p$ is defined as the homotopy class of $\gamma_1\gamma_2$.\footnote{Note that by definition, $p$ sends the one-step path $C'\to C$ in $\D$ to a path from $a_{C'}$ to $a_C$ that passes through a point in $X_{C',C}$ instead, which is \textit{not} homotopy equivalent to the opposite path of $\gamma_1\gamma_2$.} 
By \cite[Theorem I.4.10]{VdL_thesis}, $p$ is well-defined and is an equivalence from $\D$ to $\pi_1(\cX_J^\CC)$. 
Since $g\cdot X_{C,C'}=X_{g\cdot C,g\cdot C'}$ for all $g \in N(W,J)$, it follows immediately from the construction that $p$ is $N(W,J)$-equivariant.
\end{proof}

We can now prove the main results of this section: Theorem \ref{thm:main} and Corollary \ref{cor:Kpi1}.
The first is an isomorphism between the short exact sequences appearing in  Lemma \ref{lem:PAW} and \eqref{eq:TitsconeSES}. Together with Corollary \ref{cor:Kpi1}, they provide the proof of Theorem \ref{thm:main1v2}.

\begin{Theorem}\label{thm:main}
Let $\Gamma$ be a spherical Coxeter diagram with associated Coxeter group $W$ and associated Artin--Tits group $A$. Let also $J$ be a subset of the vertex set $\Gamma_0$. Fix $x\in \cX_J^\CC$ and denote also by $x$ its image in $\cX_J^\CC/N(W,J)$. 

We have an isomorphism of short exact sequences
\begin{equation}\label{eqn:isomorphism_sec}\begin{tikzcd}
	1 & {N(P,J)} & {N(A,J)} & {N(W,J)} & 1 \\
	1 & {\pi_1(\cX_J^\CC,x)} & {\pi_1(\cX_J^\CC/N(W,J),x)} & {N(W,J)} & 1
	\arrow[from=1-1, to=1-2]
	\arrow[from=1-2, to=1-3]
	\arrow["\simeq", from=1-2, to=2-2]
	\arrow[from=1-3, to=1-4]
	\arrow["{\simeq }", from=1-3, to=2-3]
	\arrow[from=1-4, to=1-5]
	\arrow[equals, from=1-4, to=2-4]
	\arrow[from=2-1, to=2-2]
	\arrow[from=2-2, to=2-3]
	\arrow[from=2-3, to=2-4]
	\arrow[from=2-4, to=2-5]
\end{tikzcd},
\end{equation}
where the first short exact sequence is induced by the canonical quotient $\pi:A \to W$ (see Lemma \ref{lem:PAW}) and the second is induced by the normal covering $\cX_J^\CC \to \cX_J^\CC/N(W,J)$.
\end{Theorem}
\begin{proof}
    By Lemma \ref{lem:equiv_quotient_pi1}, $p: \D \to \pi_1(\cX_J^\CC)$ and $\overline{p}: \D/N(W,J) \to \pi_1(\cX_J^\CC/N(W,J))$ provide an equivalence of groupoid coverings between $\cQ: \D \to \D/N(W,J)$ and $\pi_1(\cX_J^\CC) \to \pi_1(\cX_J^\CC/N(W,J))$, which gives us the following isomorphism of short exact sequences:
    \begin{equation}\label{eq:isomses1}\begin{tikzcd}[ampersand replacement=\&]
	1 \& {\D_J} \& {(\D/N(W,J))_J} \& {N(W,J)} \& 1 \\
    1 \& {\pi_1(\cX_J^\CC,x)} \& {\pi_1(\cX_J^\CC/N(W,J),x)} \& {N(W,J)} \& 1
	\arrow[from=1-1, to=1-2]
	\arrow["\cQ", from=1-2, to=1-3]
	\arrow["p", from=1-2, to=2-2]
	\arrow[from=1-3, to=1-4]
	\arrow["\overline{p}", from=1-3, to=2-3]
	\arrow[from=1-4, to=1-5]
	\arrow[equals, from=1-4, to=2-4]
	\arrow[from=2-1, to=2-2]
	\arrow[from=2-2, to=2-3]
	\arrow[from=2-3, to=2-4]
	\arrow[from=2-4, to=2-5]
    \end{tikzcd}\end{equation}
    where $x$ denotes the image of $(e,J)$ under $p$, and $\D_J$ signifies the vertex group based at $(e,J)$ for notation simplicity. 
    
    By Corollary \ref{cor:GandQequiv}, $\cQ$ is isomorphic to the groupoid covering $\cG: \D \to \R$ with $N(W,J)$ as its deck transformation group, so the top short exact sequence in \eqref{eq:isomses1} is isomorphic to 
    \begin{align}\label{eq:ses7}
        1 \to \D_J \xrightarrow{\cG} \R_J \xrightarrow{\psi} N(W,J) \to 1.
    \end{align}
    Note that each $\mu \in \R_J$ lifts to a morphism in $\D_J$ with source $(e,J)$ and target $(w,J)$, where $w=\pi(\mu)$. Hence $\psi(\mu)=\ol{w}$ by Lemma \ref{lem:normaliseraction} (see \eqref{eq:sesfornormalcover}).
    By Theorem \ref{thm:Artinnormaliser}, the canonical quotient $q:\mathrm{Norm}_A(A_J) \twoheadrightarrow N(A,J)$ restricts to an isomorphism from $\R_J\subseteq \mathrm{Norm}_A(A_J)$ to $N(A,J)$, so \eqref{eq:ses7} can be rewritten as
    \begin{align}\label{eq:ses8}
1 \to \D_J \to N(A,J) \xrightarrow{\phi} N(W,J) \to 1.
    \end{align}
    where $\phi \circ q|_{\R_J}=\psi$.
    It follows that $\phi$ agrees with $\ol{\pi}$ in Lemma \ref{lem:PAW}, and we obtain from it that \eqref{eq:ses8} is isomorphic to the top short exact sequence in \eqref{eqn:isomorphism_sec}. Composing the chain of isomorphisms we obtained throughout the proof finally yields the desired isomorphism in \eqref{eqn:isomorphism_sec}.
\end{proof}

\begin{Remark}
    With the notation of Theorem \ref{thm:main}, the second vertical map in \eqref{eqn:isomorphism_sec} is obtained by composing the following sequence of equivalences
\begin{equation}\label{eqn:sequence_equivalences}
\begin{tikzcd}
	{N(A,J)} & \sR & {\D/N(W,J)} & {\pi_1(\cX_J^\CC/N(W,J))} & {\pi_1(\cX_J^\CC/N(W,J),x).}
	\arrow[from=1-1, to=1-2]
	\arrow["{\cG^{'{-1}}}", from=1-2, to=1-3]
	\arrow["{\overline{p}}", from=1-3, to=1-4]
	\arrow[from=1-4, to=1-5]
\end{tikzcd}
\end{equation}
where the first (resp.\ second, third) equivalence comes from Theorem \ref{thm:Artinnormaliser} (resp.\ Proposition \ref{prop:G'equiv}, Lemma \ref{lem:equiv_quotient_pi1}). The final equivalence is just taking the vertex group of a groupoid.
\end{Remark}

\begin{Corollary}\label{cor:Kpi1}
    The space $\cX_J^\CC/N(W,J)$ is a $K(\pi,1)$ for $N(A,J)$.
\end{Corollary}

\begin{proof}
Note that $\cX_J^\CC$ is a complexified hyperplane complement from a real, central (since the hyperplanes are linear) finite hyperplane arrangement in a vector space (recall that we are in spherical-type).
Moreover, the chambers of the real arrangement are open simplicial cones. 
Therefore Deligne's Theorem \cite{Del72} applies, and $\cX_J^\CC$ is a $K(\pi,1)$. 
Since $\cX_J^\CC \to \cX_J^\CC/N(W,J)$ is a normal cover (see paragraph before \eqref{eq:TitsconeSES}) and  $\pi_1(\cX_J^\CC/N(W,J),x) \cong N(A,J)$ by Theorem \ref{thm:main}, the result follows.
\end{proof}

\begin{Remark} \label{rmk:finiteCWcomplex}
    Recall from Remark \ref{rmk:salvetticomplex} that the complexified complement $\cX_J^\CC$ has a finite CW-model given by the Salvetti complex. By the above corollary, this is also a CW-model for a classifying space of $N(P,J)$. Moreover, as in the proof of \cite[Theorem 3.1]{Paris_Kpi1}, one can check that this CW complex is compatible with the action of $N(W,J)$, so we can obtain from it a finite CW-model for $\cX_J^\CC/N(W,J)$ which is a classifying space of $N(A,J)$.
\end{Remark}

\section{$D_4$-Example} \label{sec:example}
In this section, we present an in-depth example with 
\[
\Gamma = D_4 = 
    \dynkin[labels={4,2,1,3},edge
length=.75cm,label directions={,right,,}] D4
\text{ and } J = \{2\}.
\]
In the setting of 3-fold flopping contractions, this setup appears for $cD_4$ singularities with three curves meeting at a point; see \cite[Example 7.4]{Wemyss_flopscluster} and \cite[Example 2.2]{HW_faithfulhyperplane}.
It also appeared in \cite[\S 8.2]{Ko25} in the setting of core (double) cosets; see also \cite[Example 1.9 \& 4.14]{Ko25}.

Let us denote by $s\coloneqq s_1, \ t \coloneqq s_2, \ u \coloneqq s_3, \ v \coloneqq s_4$ the simple reflections of $W=W_\Gamma$ (so $t$ correspond to the trivalent vertex). 
Let $V$ denote the standard geometric representation of $W$.
The positive roots associated to $W$ are given as follows:
\begin{align*}
\Phi^+ = \{&\alpha_1,\ \alpha_2,\ \alpha_3,\ \alpha_4,\ \alpha_1+\alpha_2,\ \alpha_3+\alpha_2,\ \alpha_4+\alpha_2, \\
    &\alpha_1+\alpha_2+\alpha_3,\ \alpha_1+\alpha_2+\alpha_4,\ \alpha_3+\alpha_2+\alpha_4, \\
    &\alpha_1+\alpha_2+\alpha_3+\alpha_4,\ \alpha_1+2\alpha_2+\alpha_3+\alpha_4\}.
\end{align*}

\subsection{Restricted arrangement in the Tits cone intersection}
Let $\Theta = V^*$ denote the contragredient (dual) representation of $W$.
Recall that for $\beta \in V$, $H_\beta$ denotes the hyperplane in $\Theta$ defined by
\[
H_\beta \coloneqq \{ f \in \Theta \mid f(\beta) = 0\} \subseteq \Theta.
\]
Then $\Theta_J = H_{\alpha_2}$ by definition, and since $W$ is finite, we have that 
\[
\Cone(J)^\circ = \Cone(J) = \Theta_J = H_{\alpha_2}.
\]

The set of positive $J$-roots is given by $\Phi^J_+ = \Phi^+\cap \Phi^J = \Phi^+ \setminus \{\alpha_2\}$, and the set of hyperplanes $\cH^J$ is given by $\{ H_\alpha \cap \Theta_J \mid \alpha \neq \alpha_2\}$.
Note that not all positive $J$-roots provide distinct hyperplanes in $\Theta_J = H_{\alpha_2}$; for example, we have $H_{\alpha_1} \cap H_{\alpha_2} = H_{\alpha_1+\alpha_2} \cap H_{\alpha_2}$.
In total, $\cH^J$ consists of 7 hyperplanes $H_\alpha \cap \Theta_J$ obtained by taking $H_\alpha$ to be one of the following:
\[
\{H_{\alpha_1}, H_{\alpha_3}, H_{\alpha_4}, H_{\alpha_1 + \alpha_2 + \alpha_3}, H_{\alpha_1+ \alpha_2+\alpha_4}, H_{\alpha_2+\alpha_3+\alpha_4}, H_{\alpha_1+\alpha_2+\alpha_3+\alpha_4}\}.
\]
The hyperplane arrangement $(\Theta_J,\cH^J)$ associated to the Tits cone intersection of $J$ has 32 chambers and 7 hyperplanes.
This arrangement is presented in Figure \ref{fig:TCI-arrangement}. 
\begin{figure}[h]
    \centering
    \includegraphics[width=0.45\linewidth]{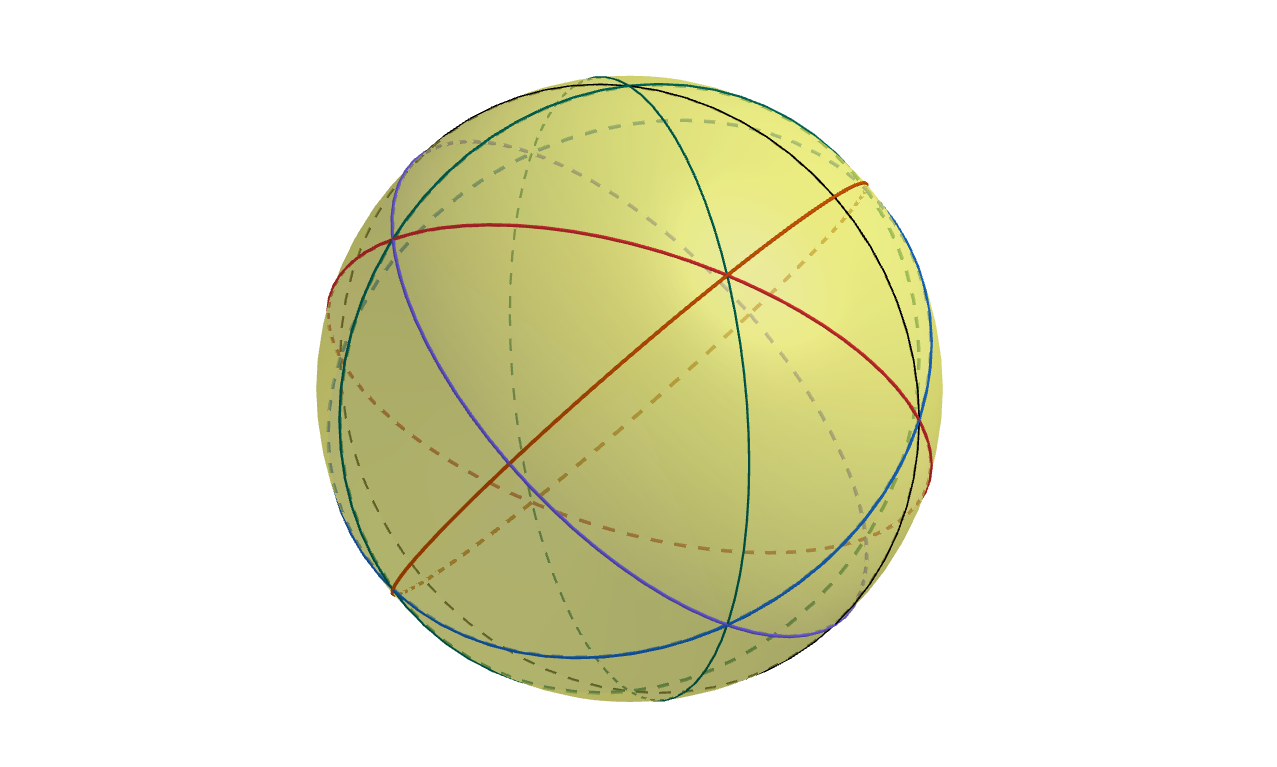}
    \includegraphics[width=0.45\linewidth]{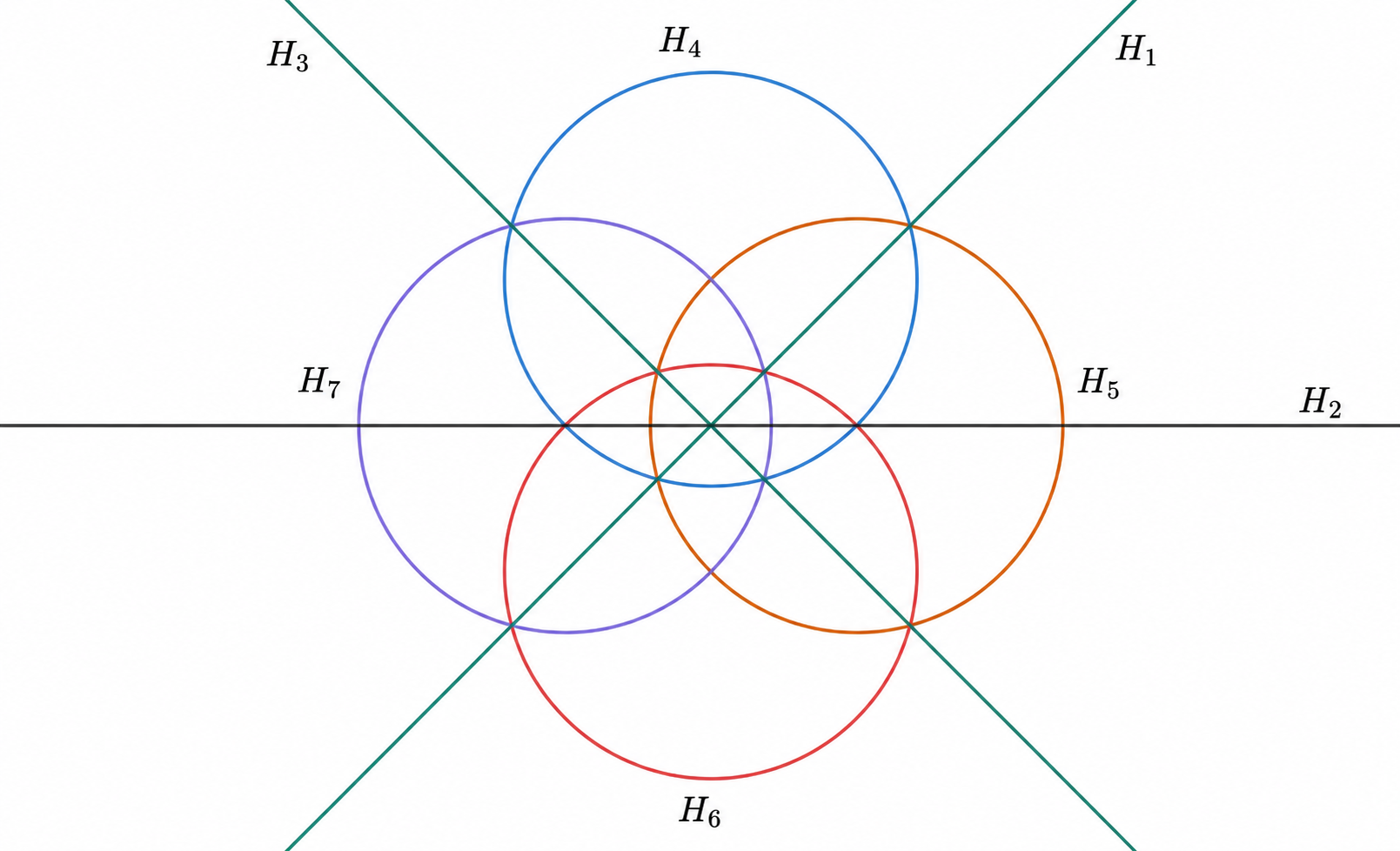}
    \caption{Hyperplane arrangement of the Tits cone intersection for type $D_4$ with $J=\{2\}$ (trivalent node). Left: the intersection of the elements of $\cH^J$ with the unit sphere $\mathbb{S}^2$. Right: the corresponding image obtained via the stereographic projection $\mathbb{S}^2\setminus \{(0,0,1)\}\to \mathbb{R}^2$. In the picture on the right, the line labels correspond to
    $H_1=H_{\alpha_1}, H_2=H_{\alpha_2+\alpha_3+\alpha_4}, H_3=H_{\alpha_1+\alpha_2+\alpha_3+\alpha_4}$ and the four circle labels correspond to
    $H_4=H_{\alpha_3}, H_5=H_{\alpha_1+\alpha_2+\alpha_3}, H_6=H_{\alpha_4}$ and $H_7=H_{\alpha_1+\alpha_2+\alpha_4}$.}
    \label{fig:TCI-arrangement}
\end{figure}
Notice that the arrangement $(\Theta_J,\cH^J)$ cannot be a reflection arrangement. Indeed, it is irreducible and we know that an irreducible reflection arrangement in dimension 3 has either 6, 9 or 15 hyperplanes (respectively for groups of type $A_3,B_3$ and $H_3$) thanks to the classification of finite real reflection groups.

\subsection{Description of the various groupoids involved}
We now give explicit descriptions of the groupoids that appear in the Diagram \eqref{eq:global_diagram}.

\subsubsection{The Brink--Howlett groupoid $\BH$ and its universal cover $\wt{\BH}$}
By Definition \ref{defn:universalgroupoid}, the groupoid $\widetilde{\BH}$ is a simply connected groupoid whose set of objects is equal to the set of chambers $\Cham(J)$ of $(\Theta_J,\cH^J)$. In other words, all of its $\Hom$ sets are singletons. 
Since $\widetilde{\BH}$ is generated by morphisms between adjacent chambers by Proposition \ref{prop:genWC}, it is completely determined by the adjacency graph of the chambers of $(\Theta_J,\cH^J)$, which is shown in Figure \ref{fig:adj_graph}.

\begin{figure}[h]
    \centering
    \includegraphics[scale=0.8]{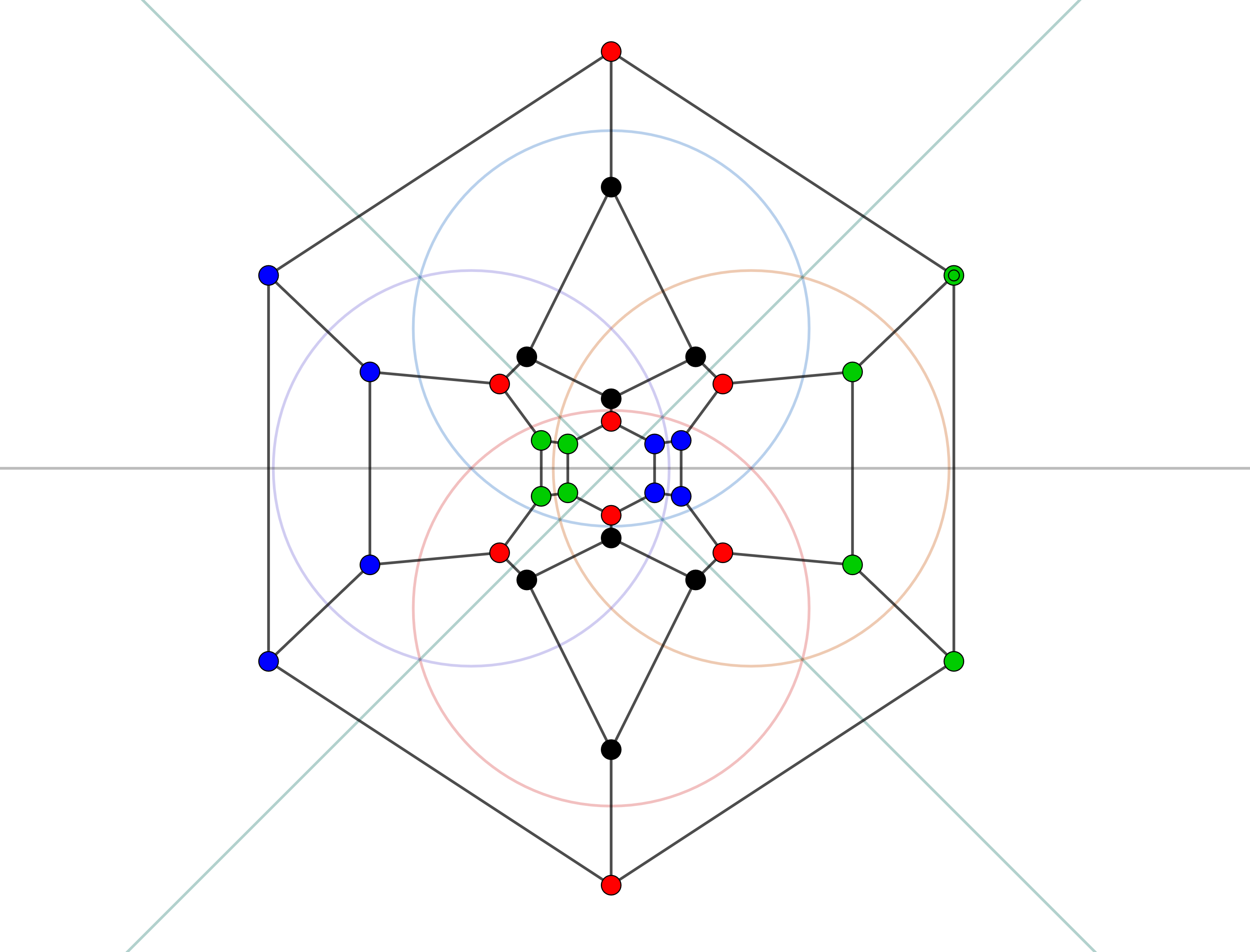}
    \caption{The adjacency graph of chambers overlaid on top of the stereographic projection of the arrangement $(\Theta_J,\cH^J)$. The colours symbolise the four orbits under the action of $N(W,J)$.}
    \label{fig:adj_graph}
\end{figure}


More precisely, recall that the chambers of $(\Theta_J,\cH^J)$ are labelled by pairs $(x,I)$ for $x \in W$ and $I \subseteq \Gamma_0$ such that $W_Jx=xW_I$ and $x$ is the minimal length representative in $xW_I$ (Theorem \ref{thm:labellingtheorem}). The bijection with chambers is given by $(x,I) \equiv xC_I$.
By Proposition \ref{prop:simplewallcrossing}, any two adjacent chambers are related by the wall-crossing formula \eqref{eqn:wallcrossingformula}, so these labels can be easily computed through an iterative application of the wall-crossing formula starting (for example) at $(e,J) \equiv C_J$.
In this way, we obtain the generating graph of $\wt{\BH}$, which is simply the graph in Figure \ref{fig:adj_graph} with vertices replaced by the corresponding chamber labels, and each arrow between adjacent chamber $(x,I)$ and $(y,K)$ labelled by 
\begin{equation} \label{eqn:wallcrossingelement}
    v_{a,I} = x^{-1}y:(x,I)\to \omega_{a,I}(x,I)=(y,K) \quad \text{and} \quad v_{b,K} = y^{-1}x :(y,K)\to \omega_{b,K}(y,K)=(x,I),
\end{equation}
This is shown in Figure \ref{fig:gengraph_tilde_bh}, with arrow labels omitted since they can be deduced from the chamber labels.

\begin{figure}
\[\begin{tikzcd}[scale cd=0.75, column sep=tiny, row sep=small,ampersand replacement=\&]
	\&\&\&\& \textcolor{rgb,255:red,247;green,59;blue,62}{{(stuts,2)}} \&\&\&\& \\
	\&\&\&\& {(stu,1)} \\
	\&\&\& {(stuv,1)} \&\& {(st,1)} \\
	\&\& \textcolor{rgb,255:red,247;green,59;blue,62}{{(stuvts,2)}} \&\& {(stv,1)} \&\& \textcolor{rgb,255:red,247;green,59;blue,62}{{(e,2)}} \\
	\textcolor{rgb,255:red,72;green,75;blue,234}{{(stutsvt,4)}} \&\&\&\& \textcolor{rgb,255:red,247;green,59;blue,62}{{(stvts,2)}} \&\&\&\& \textcolor{rgb,255:red,22;green,187;blue,60}{{(uts,3)}} \\
	\& \textcolor{rgb,255:red,72;green,75;blue,234}{{(stutsvts,4)}} \& \textcolor{rgb,255:red,22;green,187;blue,60}{{(stuvtsut,3)}} \& \textcolor{rgb,255:red,22;green,187;blue,60}{{(stvtsut,3)}} \&\& \textcolor{rgb,255:red,72;green,75;blue,234}{{(vts,4)}} \& \textcolor{rgb,255:red,72;green,75;blue,234}{{(vt,4)}} \& \textcolor{rgb,255:red,22;green,187;blue,60}{{(ut,3)}} \\
	\& \textcolor{rgb,255:red,72;green,75;blue,234}{{(stutsvtsu,4)}} \& \textcolor{rgb,255:red,22;green,187;blue,60}{{(stuvtsutv,3)}} \& \textcolor{rgb,255:red,22;green,187;blue,60}{{(stvtsutv,3)}} \&\& \textcolor{rgb,255:red,72;green,75;blue,234}{{(vtsu,4)}} \& \textcolor{rgb,255:red,72;green,75;blue,234}{{(vtu,4)}} \& \textcolor{rgb,255:red,22;green,187;blue,60}{{(utv,3)}} \\
	\textcolor{rgb,255:red,72;green,75;blue,234}{{(stutsvtu,4)}} \&\&\&\& \textcolor{rgb,255:red,247;green,59;blue,62}{{(vtsutv,2)}} \&\&\&\& \textcolor{rgb,255:red,22;green,187;blue,60}{{(utsv,3)}} \\
	\&\& \textcolor{rgb,255:red,247;green,59;blue,62}{{(stuvtsutvtu,2)}} \&\& {(vtutvstv,1)} \&\& \textcolor{rgb,255:red,247;green,59;blue,62}{{(vtutv,2)}} \\
	\&\&\& {(utsvtustv,1)} \&\& {(vtutvst,1)} \\
	\&\&\&\& {(utsvtust,1)} \\
	\&\&\&\& \textcolor{rgb,255:red,247;green,59;blue,62}{{(utsvtu,2)}}
	\arrow[shift right, from=1-5, to=2-5]
	\arrow[shift right, from=1-5, to=5-1]
	\arrow[shift right, from=1-5, to=5-9]
	\arrow[shift right, from=2-5, to=1-5]
	\arrow[shift right, from=2-5, to=3-4]
	\arrow[shift right, from=2-5, to=3-6]
	\arrow[shift right, from=3-4, to=2-5]
	\arrow[shift right, from=3-4, to=4-3]
	\arrow[shift right, from=3-4, to=4-5]
	\arrow[shift right, from=3-6, to=2-5]
	\arrow[shift right, from=3-6, to=4-5]
	\arrow[shift right, from=3-6, to=4-7]
	\arrow[shift right, from=4-3, to=3-4]
	\arrow[shift right, from=4-3, to=6-2]
	\arrow[shift right, from=4-3, to=6-3]
	\arrow[shift right, from=4-5, to=3-4]
	\arrow[shift right, from=4-5, to=3-6]
	\arrow[shift right, from=4-5, to=5-5]
	\arrow[shift right, from=4-7, to=3-6]
	\arrow[shift right, from=4-7, to=6-7]
	\arrow[shift right, from=4-7, to=6-8]
	\arrow[shift right, from=5-1, to=1-5]
	\arrow[shift right, from=5-1, to=6-2]
	\arrow[shift right, from=5-1, to=8-1]
	\arrow[shift right, from=5-5, to=4-5]
	\arrow[shift right, from=5-5, to=6-4]
	\arrow[shift right, from=5-5, to=6-6]
	\arrow[shift right, from=5-9, to=1-5]
	\arrow[shift right, from=5-9, to=6-8]
	\arrow[shift right, from=5-9, to=8-9]
	\arrow[shift right, from=6-2, to=4-3]
	\arrow[shift right, from=6-2, to=5-1]
	\arrow[shift right, from=6-2, to=7-2]
	\arrow[shift right, from=6-3, to=4-3]
	\arrow[shift right, from=6-3, to=6-4]
	\arrow[shift right, from=6-3, to=7-3]
	\arrow[shift right, from=6-4, to=5-5]
	\arrow[shift right, from=6-4, to=6-3]
	\arrow[shift right, from=6-4, to=7-4]
	\arrow[shift right, from=6-6, to=5-5]
	\arrow[shift right, from=6-6, to=6-7]
	\arrow[shift right, from=6-6, to=7-6]
	\arrow[shift right, from=6-7, to=4-7]
	\arrow[shift right, from=6-7, to=6-6]
	\arrow[shift right, from=6-7, to=7-7]
	\arrow[shift right, from=6-8, to=4-7]
	\arrow[shift right, from=6-8, to=5-9]
	\arrow[shift right, from=6-8, to=7-8]
	\arrow[shift right, from=7-2, to=6-2]
	\arrow[shift right, from=7-2, to=8-1]
	\arrow[shift right, from=7-2, to=9-3]
	\arrow[shift right, from=7-3, to=6-3]
	\arrow[shift right, from=7-3, to=7-4]
	\arrow[shift right, from=7-3, to=9-3]
	\arrow[shift right, from=7-4, to=6-4]
	\arrow[shift right, from=7-4, to=7-3]
	\arrow[shift right, from=7-4, to=8-5]
	\arrow[shift right, from=7-6, to=6-6]
	\arrow[shift right, from=7-6, to=7-7]
	\arrow[shift right, from=7-6, to=8-5]
	\arrow[shift right, from=7-7, to=6-7]
	\arrow[shift right, from=7-7, to=7-6]
	\arrow[shift right, from=7-7, to=9-7]
	\arrow[shift right, from=7-8, to=6-8]
	\arrow[shift right, from=7-8, to=8-9]
	\arrow[shift right, from=7-8, to=9-7]
	\arrow[shift right, from=8-1, to=5-1]
	\arrow[shift right, from=8-1, to=7-2]
	\arrow[shift right, from=8-1, to=12-5]
	\arrow[shift right, from=8-5, to=7-4]
	\arrow[shift right, from=8-5, to=7-6]
	\arrow[shift right, from=8-5, to=9-5]
	\arrow[shift right, from=8-9, to=5-9]
	\arrow[shift right, from=8-9, to=7-8]
	\arrow[shift right, from=8-9, to=12-5]
	\arrow[shift right, from=9-3, to=7-2]
	\arrow[shift right, from=9-3, to=7-3]
	\arrow[shift right, from=9-3, to=10-4]
	\arrow[shift right, from=9-5, to=8-5]
	\arrow[shift right, from=9-5, to=10-4]
	\arrow[shift right, from=9-5, to=10-6]
	\arrow[shift right, from=9-7, to=7-7]
	\arrow[shift right, from=9-7, to=7-8]
	\arrow[shift right, from=9-7, to=10-6]
	\arrow[shift right, from=10-4, to=9-3]
	\arrow[shift right, from=10-4, to=9-5]
	\arrow[shift right, from=10-4, to=11-5]
	\arrow[shift right, from=10-6, to=9-5]
	\arrow[shift right, from=10-6, to=9-7]
	\arrow[shift right, from=10-6, to=11-5]
	\arrow[shift right, from=11-5, to=10-4]
	\arrow[shift right, from=11-5, to=10-6]
	\arrow[shift right, from=11-5, to=12-5]
	\arrow[shift right, from=12-5, to=8-1]
	\arrow[shift right, from=12-5, to=8-9]
	\arrow[shift right, from=12-5, to=11-5]
\end{tikzcd}\]

  \caption{Generating graph for the universal groupoid $\widetilde{\BH}$ associated to the Tits cone intersection for type $D_4$ with $J=\{2\}$ (trivalent node), with arrow labels omited (deducible from the chamber labels). Cf.\ Figure \ref{fig:adj_graph}.}
    \label{fig:gengraph_tilde_bh}
\end{figure}
Being a simply-connected groupoid, all paths from one chamber to another lead to the same morphism.
In particular, all two cycles $(x,I) \rightleftarrows (y,K)$ result in the identity, and the arrows are labelled mutually inverse elements (which are not necessarily equal as elements in $W$).  


By Theorem \ref{thm:WCuniversalcover}, $\widetilde{\BH}$ is the universal cover of the Brink--Howlett groupoid $\BH$ with deck transformation group $N(W,J)$, where the functor sends $(x,I) \mapsto I$ on objects and is the identity on morphism.
By Lemma \ref{lem:normaliseraction}, the action of each $g \in \mathrm{Norm}_W(W_J)$ is free on $\Cham(J)$ (the objects of $\wt{\BH}$), with chambers sharing the same second label lying in the same orbit. We have colour-coded the vertices of the graph in Figure \ref{fig:gengraph_tilde_bh} (and also Figure \ref{fig:adj_graph}) so that two chambers $(x,I)$ and $(y,K)$ share the same colour if and only if $I=K$. In particular, we see that there are 4 orbits of $N(W,J)$ (one for each $J$-associate) and each orbit contains 8 chambers.

Through this, we can obtain the generating graph of $\BH$ by taking the quotient of the generating graph of $\wt{\BH}$ in Figure \ref{fig:gengraph_tilde_bh} that identifies all vertices of the same colour.
This is depicted in Figure \ref{fig:BHgen}.
As before, notice that two opposite directions of an edge are labelled by elements in $W$ that are mutual inverses but need not be the same element. 
\begin{figure}
    \centering
\begin{tikzcd}
	&& 1 \\
	\textcolor{rgb,255:red,63;green,81;blue,243}{4} & \textcolor{rgb,255:red,249;green,57;blue,57}{2} \\
	&& \textcolor{rgb,255:red,28;green,181;blue,28}{3}
	\arrow["u", from=1-3, to=1-3, loop, in=330, out=30, distance=5mm]
	\arrow["v", from=1-3, to=1-3, loop, in=60, out=120, distance=5mm]
	\arrow["ts"', curve={height=6pt}, from=1-3, to=2-2]
	\arrow["s", from=2-1, to=2-1, loop, in=105, out=165, distance=5mm]
	\arrow["u", from=2-1, to=2-1, loop, in=195, out=255, distance=5mm]
	\arrow["tv", curve={height=-6pt}, from=2-1, to=2-2]
	\arrow["st"', curve={height=6pt}, from=2-2, to=1-3]
	\arrow["vt", curve={height=-6pt}, from=2-2, to=2-1]
	\arrow["ut", curve={height=-6pt}, from=2-2, to=3-3]
	\arrow["tu", curve={height=-6pt}, from=3-3, to=2-2]
	\arrow["s", from=3-3, to=3-3, loop, in=330, out=30, distance=5mm]
	\arrow["v", from=3-3, to=3-3, loop, in=240, out=300, distance=5mm]
\end{tikzcd}
    \caption{Generating graph of the Brink--Howlett groupoid $\BH$ associated to the Tits cone intersection for type $D_4$ with $J=\{2\}$ (trivalent node).}
    \label{fig:BHgen}
\end{figure}

Being a (universal) cover, each morphism (path) $w \in \Hom_{\BH}(I,K)$ has a unique lift to a morphism (path) $w$ in $\widetilde{\BH}$ once we fix a lift $(x,I) \in \Cham(J)$ of the source $I$.
As such, each morphism $x:J \to I$ in $\BH$ is uniquely lifted to the morphism $x: (e,J) \to (x,I)$ in $\wt{\BH}$.
In particular, $N(W,J)$ is isomorphic to the vertex group $\BH_J$ of $\BH$ given by the 8-element subgroup of $W$: \{$x \in W \mid (x,2) \in \Cham(J)\}$ (i.e. the red vertices in Figure \ref{fig:gengraph_tilde_bh}).

Suppose that two morphisms in $\BH$ are related by the atomic braid relations (Definition \ref{def:atomicbraidrel}). Then in $\wt{\BH}$ these are lifted to morphisms which are related by codimension two relations (Definition \ref{def:codim2rel}).
In our example, these correspond to paths through the polygonal faces -- including the outer-most hexagon -- of the generating graph for $\wt{\BH}$ in Figure \ref{fig:gengraph_tilde_bh}.
For instance, the following atomic braid relations in $\BH$
\[
st\cdot v\cdot ts = vt\cdot s \cdot tv \quad \text{and} \quad s\cdot tv\cdot st = tv\cdot ts\cdot v
\]
correspond, respectively, to the following paths along the same hexagonal face: 
\begin{equation} \label{eqn:hexagonatomicbraid}
\begin{tikzcd}[sep=tiny]
	& {(st,1)} & {(stv,1)} &&&& \textcolor{rgb,255:red,59;green,69;blue,247}{{(vt,4)}} & \textcolor{rgb,255:red,255;green,51;blue,58}{{(e,2)}} & \\
	\textcolor{rgb,255:red,255;green,51;blue,58}{{(e,2)}} &&& \textcolor{rgb,255:red,255;green,51;blue,58}{{(stvts,2)}}  &\text{and}& \textcolor{rgb,255:red,59;green,69;blue,247}{{(vts,4)}} &&& {(st,1)}. \\
	& \textcolor{rgb,255:red,59;green,69;blue,247}{{(vt,4)}} & \textcolor{rgb,255:red,59;green,69;blue,247}{{(vts,4)}} &&&& \textcolor{rgb,255:red,255;green,51;blue,58}{{(stvts,2)}} & {(stv,1)}
	\arrow[from=1-2, to=1-3]
	\arrow[from=1-3, to=2-4]
	\arrow[from=1-7, to=1-8]
	\arrow[from=1-8, to=2-9]
	\arrow[from=2-1, to=1-2]
	\arrow[from=2-1, to=3-2]
	\arrow[from=2-6, to=1-7]
	\arrow[from=2-6, to=3-7]
	\arrow[from=3-2, to=3-3]
	\arrow[from=3-3, to=2-4]
	\arrow[from=3-7, to=3-8]
	\arrow[from=3-8, to=2-9]
\end{tikzcd}
\end{equation}

In particular, the atomic braid relations for our example are all of length 2 (squares) or length 3 (hexagons).
Since the defining relations of $\BH$ are given by the atomic braid relations and the quadratic relations (Theorem \ref{thm:BHthmA}), $\wt{\BH}$ is also isomorphic to the fundamental groupoid of the space obtained from the generating graph by attaching two-cells to each two cycles between adjacent chambers, and also to the polygonal faces including the outer-most hexagon (equivalently, the space obtained from the adjacency graph by attaching two-cells to polygonal faces).
For our example, the resulting space is a 2-sphere.

One should view $\widetilde{\BH}$ as the (right) ``Cayley graph'' for the groupoid $\BH$ with respect to the generators $v_{a,I}$. 
Indeed, when $J = \emptyset$, the arrangement is just the Coxeter arrangement, and the generating graph of $\widetilde{\BH}$ is exactly the  Cayley graph of $W = \BH$ with respect to the simple generators $s_a=v_{a,\varnothing}$.

\subsubsection{The Deligne groupoid $\sD$ and the reduced ribbon groupoid $\R$}
The Deligne groupoid and the reduced ribbon groupoid are essentially defined from the same generating graphs of $\wt{\BH}$ (Figure \ref{fig:gengraph_tilde_bh}) and $\BH$ (Figure \ref{fig:BHgen}) respectively, where the quadratic relations are removed -- just like defining relations of the Artin--Tits group $A$ and its Coxeter group $W$ simply differs by removing the quadratic relations.
Indeed, the Deligne groupoid $\D$ is by definition the groupoid completion of the category $\D^+$ whose morphisms are paths of the graph obtained from Figure \ref{fig:adj_graph} with each edge $x$ --- $y$ replaced by a pair of arrows $x\leftrightarrows y$ (see \S\ref{sec:Dgroupoid}), modulo relations generated by identifying geodesic paths with the same source and target.
This is the same graph as Figure \ref{fig:gengraph_tilde_bh} -- the point is that the quadratic relations no longer hold, so that morphisms through two-cycles no longer result in the identity.

Note that by Corollary \ref{cor:arrangement_atomicMat}, it is sufficient to identify only the geodesic paths that come from codimension two relations (Definition \ref{def:codim2rel}) -- these correspond exactly to the polygonal faces (squares and hexagons) discussed in the previous subsection for $\wt{\BH}$.
For instance, the diagrams of morphisms in \eqref{eqn:hexagonatomicbraid} -- now viewed as morphisms in $\D^+$ (and $\D$) -- are also commutative.


By Corollary \ref{cor:GandQequiv}, $\D$ is a normal covering of the reduced ribbon groupoid with deck transformation group $N(W,J)$.
We refer the reader to \S\ref{sec:mainresult} for the definition of the covering functor $\cG: \D \to \R$, but the important point here is that we can describe $\R$ as $\D/N(W,J)$.
In particular, since the action of $N(W,J)$ on $\D$ is essentially the same as that on $\wt{\BH}$ (see Remark \ref{rmk:NormactiononD}), we can describe $\R$ via the $N(W,J)$-quotient of the generating graph of $\sD$: the generating graph of $\R$ is simply the graph in Figure \ref{fig:BHgen}, with arrow labels in terms of the generators $s,t,u,v \in W$ replaced by their respective positive lifts $\sigma_s,\sigma_t, \sigma_u, \sigma_v \in A^+$ (Lemma \ref{lem:Gwelldef}).
As before, the relations in $\R$ are also reflected in $\D$ since $\D$ is the normal covering of $\R$, so all relations in $\R$ are generated by those induced from the polygonal faces of the graph.
We emphasise again that while two-cycles (and looping twice) in the generating graph of $\BH$ result in the identity, this is no longer the case in $\R$ ($s^2=e$, but $\sigma_s^2\neq e$).

The functor $\sD\to \widetilde{\BH}$ is easy to describe: it is the identity on objects and it sends a path $p:(x,I) \to (y,K)$ in $\D$ to the (unique) morphism $x^{-1}y:(x,I) \to (y,K)$ in $\widetilde{\BH}$, which agrees with the morphism $v_p$ induced by the path $p$. 
The natural functor $\sR\to \BH$ is also simply induced by the canonical quotient map $A\to W$.

We end this example with the following remark.
\begin{Remark}
The vertices of the graph in Figures \ref{fig:adj_graph} and \ref{fig:gengraph_tilde_bh} can be graded according to the length of a (equivalently, all) geodesic path between the chamber labelling it and the chamber $C\equiv {\color{red}(e,2)}$. Equivalently, the grading of a vertex labelled by $(x,I)$  is the standard expression length (Definition \ref{def:stdexp}) of $x$ (Lemma \ref{lem:stdexp-geodesic}). Rearranging the graph according to this grading results in the following, where we represent only the colour of each chamber (the colours together with the fact that the bottom element is the chamber ${\color{red}(e,2)}$ completely determine the graph).
\[\begin{tikzcd}[sep=tiny]
	&&& \textcolor{rgb,255:red,247;green,59;blue,62}{\bullet} &&& \\
	& \textcolor{rgb,255:red,22;green,187;blue,60}{\bullet} && \textcolor{rgb,255:red,72;green,75;blue,234}{\bullet} && \bullet \\
	\textcolor{rgb,255:red,22;green,187;blue,60}{\bullet} & \textcolor{rgb,255:red,22;green,187;blue,60}{\bullet} & \textcolor{rgb,255:red,72;green,75;blue,234}{\bullet} && \textcolor{rgb,255:red,72;green,75;blue,234}{\bullet} & \bullet & \bullet \\
	\textcolor{rgb,255:red,22;green,187;blue,60}{\bullet} & \textcolor{rgb,255:red,247;green,59;blue,62}{\bullet} & \textcolor{rgb,255:red,247;green,59;blue,62}{\bullet} & \textcolor{rgb,255:red,72;green,75;blue,234}{\bullet} && \textcolor{rgb,255:red,247;green,59;blue,62}{\bullet} & \bullet \\
	\bullet & \textcolor{rgb,255:red,247;green,59;blue,62}{\bullet} && \textcolor{rgb,255:red,72;green,75;blue,234}{\bullet} & \textcolor{rgb,255:red,247;green,59;blue,62}{\bullet} & \textcolor{rgb,255:red,247;green,59;blue,62}{\bullet} & \textcolor{rgb,255:red,22;green,187;blue,60}{\bullet} \\
	\bullet & \bullet & \textcolor{rgb,255:red,72;green,75;blue,234}{\bullet} && \textcolor{rgb,255:red,72;green,75;blue,234}{\bullet} & \textcolor{rgb,255:red,22;green,187;blue,60}{\bullet} & \textcolor{rgb,255:red,22;green,187;blue,60}{\bullet} \\
	& \bullet && \textcolor{rgb,255:red,72;green,75;blue,234}{\bullet} && \textcolor{rgb,255:red,22;green,187;blue,60}{\bullet} \\
	&&& \textcolor{rgb,255:red,247;green,59;blue,62}{\bullet}
	\arrow[no head, from=1-4, to=2-2]
	\arrow[no head, from=1-4, to=2-4]
	\arrow[no head, from=2-2, to=3-1]
	\arrow[no head, from=2-4, to=3-3]
	\arrow[no head, from=2-4, to=3-5]
	\arrow[no head, from=2-6, to=1-4]
	\arrow[no head, from=2-6, to=3-7]
	\arrow[no head, from=3-2, to=2-2]
	\arrow[no head, from=3-2, to=4-1]
	\arrow[no head, from=3-3, to=4-2]
	\arrow[no head, from=3-3, to=4-4]
	\arrow[no head, from=3-5, to=4-4]
	\arrow[no head, from=3-5, to=4-6]
	\arrow[no head, from=3-6, to=2-6]
	\arrow[no head, from=3-7, to=4-6]
	\arrow[no head, from=3-7, to=4-7]
	\arrow[no head, from=4-1, to=3-1]
	\arrow[no head, from=4-1, to=5-2]
	\arrow[no head, from=4-2, to=3-1]
	\arrow[no head, from=4-2, to=5-1]
	\arrow[no head, from=4-3, to=3-2]
	\arrow[no head, from=4-3, to=3-6]
	\arrow[no head, from=4-3, to=5-4]
	\arrow[no head, from=4-4, to=5-5]
	\arrow[no head, from=4-6, to=5-7]
	\arrow[no head, from=4-7, to=3-6]
	\arrow[no head, from=4-7, to=5-6]
	\arrow[no head, from=5-1, to=6-2]
	\arrow[no head, from=5-2, to=6-1]
	\arrow[no head, from=5-2, to=6-3]
	\arrow[no head, from=5-4, to=6-5]
	\arrow[no head, from=6-1, to=5-1]
	\arrow[no head, from=6-1, to=7-2]
	\arrow[no head, from=6-2, to=5-5]
	\arrow[no head, from=6-3, to=5-4]
	\arrow[no head, from=6-3, to=7-4]
	\arrow[no head, from=6-5, to=5-6]
	\arrow[no head, from=6-5, to=7-4]
	\arrow[no head, from=6-6, to=5-5]
	\arrow[no head, from=6-6, to=5-7]
	\arrow[no head, from=6-7, to=5-6]
	\arrow[no head, from=6-7, to=5-7]
	\arrow[no head, from=7-2, to=6-2]
	\arrow[no head, from=7-2, to=8-4]
	\arrow[no head, from=7-4, to=8-4]
	\arrow[no head, from=7-6, to=6-6]
	\arrow[no head, from=7-6, to=6-7]
	\arrow[no head, from=8-4, to=7-6]
\end{tikzcd}\]
This corresponds to the Hasse diagram for the poset of regions with base chamber $C \equiv {\color{red}(e,2)}$ (see \cite[\S I]{Edelman_regionposet} or \cite[Definition 2.2]{Delucchi_combinatorialrmk}), where the minimum element is $C$ and the maximum element is the antipodal chamber $-C \equiv {\color{red}(stuvtsutvtu,2)}$. 
In particular, this is also the Hasse diagram for the divisibility relation for the set of simple morphisms with source $C$ for the Garside structure of the Deligne groupoid $\D$. 
It also agrees with the Hasse diagram (with opposite orientation) given above Remark 8.6 in \cite{Ko25}.
\end{Remark}

\bibliographystyle{alpha}
\bibliography{monbib}

\end{document}